\documentclass[11pt,letterpaper]{amsart}
\usepackage[english]{babel}
\usepackage{amsmath,amsthm,amssymb,fourier,mathrsfs,enumerate,color,ifpdf}
\usepackage{tikz,tikz-cd}
\usetikzlibrary{decorations.pathmorphing,arrows,intersections}
\usepackage{enumitem}
\usepackage{float}
\usepackage{csquotes}
\usepackage[
  backend=biber,
  style=alphabetic,
  sorting=nyt,
  maxcitenames=4,
  maxalphanames=4,
  minalphanames=4
]{biblatex}
\usepackage[colorlinks=true,
            linkcolor=red,
            citecolor=blue,
            urlcolor=blue]{hyperref}
\usepackage{listings}
\usepackage{xcolor}

\lstdefinelanguage{Sage}{
  language=Python,
  morekeywords={var,plot,matrix,Graph,graphs},
  sensitive=true
}

\DeclareLabelalphaTemplate{
  \labelelement{
    \field[strwidth=1,strside=left,uppercase]{labelname}
  }
  \labelelement{
    \field[strwidth=2,strside=right]{year}
  }
}

\DeclareCiteCommand{\cite}[\mkbibbrackets]
  {\usebibmacro{prenote}}
  {%
    \usebibmacro{citeindex}%
    \bibhyperref{%
      \printfield{labelalpha}%
      \printfield{extraalpha}%
    }%
  }
  {\addcomma\space}
  {\usebibmacro{postnote}}
\defbibenvironment{bibliography}
  {\list
     {\mkbibbrackets{\printfield{labelalpha}%
        \printfield{extraalpha}}} % left label
     {\setlength{\labelwidth}{4.5em}%
      \setlength{\leftmargin}{\labelwidth}%
      \addtolength{\leftmargin}{\labelsep}%
      \setlength{\itemsep}{\bibitemsep}%
      \setlength{\parsep}{\bibparsep}}%
  }
  {\endlist}
  {\item}

\DeclareBibliographyDriver{article}{%
  \printnames{author}%
  \newunit\newblock
  \printfield{title}%
  \newunit\newblock
  \printfield{journaltitle}%
  \setunit*{\addspace}%
  \printfield{volume}%
  \setunit{\addcolon\space}%
  \printfield{number}%
  \setunit{\addcomma\space}%
  \printfield{year}%
  \newunit\newblock
  \printfield{pages}%
  \newunit\newblock
  \printfield{doi}%
  \finentry
}

\DeclareBibliographyDriver{incollection}{%
  \printnames{author}%
  \newunit\newblock
  \printfield{title}% chapter title
  \newunit\newblock
  \printtext[emph]{\printfield{booktitle}}% book title italic
  \printtext[emph]{\printfield{note}}
  \newunit\newblock
  \printnames{editor}%
  \newunit\newblock
  \printlist{publisher}%
  \setunit{\addspace}%
  \printfield{year}%
  \newunit\newblock
  \printfield{pages}%
  \newunit\newblock
  \printfield{doi}%
  \finentry
}

\usepackage{caption}
\usepackage{subcaption}
\usepackage{comment}

\newcommand{\GL}{\operatorname{GL}}

\newcommand{\SL}{\operatorname{SL}}

\newcommand{\rk}{\operatorname{rk}}

\newcommand{\setdef}[2]{ \left\{ {#1}\ : \ {#2} \right\} }

\newcommand\blfootnote[1]{%
  \begingroup
  \renewcommand\thefootnote{}\footnote{#1}%
  \addtocounter{footnote}{-1}%
  \endgroup
}
\newtheorem{theorem}{Theorem}[section]
\newtheorem{thm}[theorem]{Theorem}
\newtheorem{mthm}{Theorem}

\newtheorem{lemma}[theorem]{Lemma}

\newtheorem{corollary}[theorem]{Corollary}

\newtheorem{proposition}[theorem]{Proposition}

\newtheorem{conjecture}[theorem]{Conjecture}

\newtheorem{defn}[theorem]{Definition}

\newtheorem{example}[theorem]{Example}

\newtheorem{remark}[theorem]{Remark}

\numberwithin{equation}{section}

\begin{document}
\title{Groups Generated by Root Unipotents: Higher-rank and rank-one}

\author{Yanlong Hao}
\address{University of Michigan, Ann Arbor, Michigan, USA}
\email{ylhao@umich.edu}
 
\date{\today}

\begin{abstract}
For
$n\geq 3$, let
\[
\Gamma(Q)=\langle E_{ij}(q_{ij}):i\neq j\rangle
\]
be the subgroup of $\SL(n,\mathbb R)$ generated by elementary matrices with nonzero rational parameters $q_{ij}$. We prove that $\Gamma(Q)$ is always $S$-arithmetic, extending classical integral-parameter results to arbitrary rational parameters. Our method is effective.

We also prove strong small-generation results for finite-index subgroups of $S$-arithmetic groups. In particular, $\SL\left(n,\mathbb Z\left[\frac1N\right]\right)$, $n\geq3$, and a broad class of groups $\SL(2,\mathcal O_{F,S})$ have arbitrarily small two-generated finite-index subgroups.

We then study the rank-one family
\[
\Gamma_q=\left\langle\begin{pmatrix}1&1\\0&1
\end{pmatrix}, \begin{pmatrix}
1&0\\q&1
\end{pmatrix}\right\rangle,\qquad q=\tfrac{s}{t}\in\mathbb Q.
\]
For $q\neq0,\pm3$, we prove that
\[
\Gamma_q=\Gamma_1^{(t)}(s)
\]
if and only if its upper-triangular subgroup strictly contains
\[\left\langle\begin{pmatrix}1&1\\0&1\end{pmatrix}\right\rangle.\]

We then apply the rank-one criterion to constructions arising from non-freeness, verifying it for $q=\tfrac{s}{t}\in(-4,4)$ with $1\leq |s|\leq21$, and obtaining infinite families from indefinite binary quadratic forms and Pell-type equations.
\end{abstract}

\maketitle
\section{Introduction}\blfootnote{The work is original without using AI. ChatGPT has been used for grammar and notation check.}
The study of subgroups generated by unipotent elements has a long history and lies at the intersection of the theory of arithmetic groups, algebraic groups, and $K$-theory. The case of integral parameters is well understood. A theorem of Tits \cite{MR424966} shows that, for $n\ge 3$, elementary matrices with arbitrary nonzero integral entries generate a subgroup of finite index in $\SL(n,\mathbb Z)$. This result was subsequently generalized to a broad class of higher-rank algebraic groups by Vaserstein \cite{MR349864}, Raghunathan \cite{MR1141802}, Venkataramana \cite{MR1306038}, and others.

The rank-one situation is markedly different. Even for integral parameters, groups generated by opposite unipotents may be thin rather than arithmetic. Thus, the distinction between higher-rank rigidity and rank-one flexibility is already visible in the classical setting.

Motivated by this circle of ideas, we consider subgroups generated by prescribed elementary matrices in $\SL(n,\mathbb{R})$. For $i\neq j$, let
\[
E_{ij}(t)=I_n+te_{ij}.
\]
Let $Q=(q_{ij})$ be a collection of nonzero rational numbers indexed by the ordered pairs $(i,j)$ with $i\neq j$. For convenience, we identify $Q$ with the corresponding off-diagonal matrix. Define
\[
\Gamma(Q)=\langle E_{ij}(q_{ij}): i\neq j\rangle.
\]

At first sight, the structure of $\Gamma(Q)$ appears to depend delicately on the choice of parameters $q_{ij}$. Nevertheless, in higher rank, the situation is unexpectedly rigid. Although $\Gamma(Q)$ is generally dense in $\SL(n,\mathbb R)$, it embeds diagonally as a lattice in a suitable product of real and $p$-adic groups. Consequently, the natural framework is that of $S$-arithmetic groups rather than ordinary arithmetic groups.

Our first result shows that this higher-rank rigidity persists for arbitrary rational parameters.

\begin{mthm}\label{thm: Higher-rank}
For every collection $Q$ of nonzero rational numbers and every $n\ge3$, there exist an explicitly computable integer $N$ and an explicitly computable diagonal matrix $g\in\GL(n,\mathbb Q)$ such that $g\Gamma(Q)g^{-1}$ is a finite-index subgroup of
\[ 
\SL\left(n,\mathbb Z\left[\tfrac{1}{N}\right]\right).
\]
In particular, $\Gamma(Q)$ is $S$-arithmetic.
\end{mthm}

Moreover, our proof yields a complete and effective description of the subgroup $g\Gamma(Q)g^{-1}$ by explicit congruence conditions. This classification is stated in Theorems~\ref{thm: high rank number fields} and~\ref{thm: classification}.
 
While the arithmeticity statement extends to a much broader class of higher-rank groups, see Theorem~\ref{thm: general-higher-rank}, the effective congruence classification established here relies heavily on the structure of $\SL(n)$.

The proof of Theorem~\ref{thm: Higher-rank} applies existing rigidity results. By Tits's theorem, $\Gamma(Q)$ contains an arithmetic lattice in $\SL(n,\mathbb R)$. We then use an idea of Venkataramana \cite[Proposition~2.3]{MR4521495}, generalized in \cites{MR4812042,MR5043742}, to deduce that $\Gamma(Q)$ is $S$-arithmetic. The congruence subgroup property then reduces the classification of $\Gamma(Q)$ to a problem in finite groups studied in \cites{MR669558,MR687835}. The method also applies to other types of higher-rank groups; see the discussion in Section~\ref{Sec: general case}.

We next give an application of Theorem~\ref{thm: Higher-rank} to the ranks of finite-index subgroups. We will return to the rank-one unipotent-generation problem afterward.

Let $\Gamma$ be a finitely generated group. The \textbf{rank} of $\Gamma$, denoted by $\rk(\Gamma)$, is the minimal cardinality of a generating set of $\Gamma$. Given a chain of finite-index subgroups
\[
\Gamma=\Gamma_0>\Gamma_1>\Gamma_2>\cdots,
\]
the \textbf{rank gradient} of $\Gamma$ with respect to the chain $(\Gamma_n)$ is
\[
\lim_{n\to\infty}\frac{\rk(\Gamma_n)-1}{[\Gamma:\Gamma_n]}.
\]
The notion of rank gradient was introduced by Lackenby \cite{MR2151608} and has since been studied extensively. Vanishing of rank gradient is known for broad classes of higher-rank lattices; see, for example, \cite{fraczyk2023poisson, MR5093985}.

For our purposes, we consider a stronger small-generation property.

\begin{defn}\label{def: small k generated}
We say that $\Gamma$ has \textbf{arbitrarily small $k$-generated finite-index subgroups} if, for every finite-index subgroup $\Gamma_1<\Gamma$, there exists a finite-index subgroup $\Gamma_2<\Gamma_1$ which can be generated by $k$ elements.
\end{defn}

For $\SL(n,\mathbb Z)$, $n\geq3$, this stronger property with sharp constant $k=2$ was established by Meiri \cite{MR3565440}. More recently, Brody, Fisher, Mj, and van Limbeek obtained analogous results for certain $S$-arithmetic lattices under the Greenberg--Shalom hypothesis \cite[Corollary 7.4]{MR5043742}. The results below give unconditional $S$-arithmetic versions.

\begin{mthm}\label{THM: SL(n Z) rank gradient}
For every $n\geq3$ and $N\geq1$,
\[
\SL\left(n,\mathbb Z\left[\tfrac1N\right]\right)
\]
has arbitrarily small $2$-generated finite-index subgroups. In particular, its rank gradient with respect to every infinite chain of finite-index normal subgroups is zero.
\end{mthm}

A similar result holds for a broad class of $S$-arithmetic groups of type $\SL(2)$.

\begin{mthm}\label{Thm: SL(2,O_F) rank gradient}
Let $F$ be a non-CM number field, and let $S$ be a finite set of places of $F$ containing all Archimedean places. Assume that $\mathcal O_{F,S}^{\times}$ is infinite. Then
\[
\SL(2,\mathcal O_{F,S})
\]
has arbitrarily small $2$-generated finite-index subgroups. In particular, its rank gradient with respect to every infinite chain of finite-index normal subgroups is zero.
\end{mthm}

For a weaker result concerning general higher-rank lattices, see Theorem~\ref{thm: gradiant rank general}.

We now return to the rank-one unipotent-generation problem in $\SL(2,\mathbb R)$. This case, however, exhibits fundamentally different behavior. In contrast to the higher-rank case, subgroups generated by opposite unipotents in $\SL(2,\mathbb R)$ need not be arithmetic and may even be free. Understanding when arithmeticity occurs therefore becomes a subtle problem.

This problem has a long history, even in the more general setting of $\SL(2,\mathbb{C})$. For $q\in \mathbb{C}$, let $\Gamma_q$ be the subgroup of $\SL(2,\mathbb{C})$ generated by the matrices
\[
A=\begin{pmatrix}
   1&1\\    0&1
\end{pmatrix},\quad B_q=\begin{pmatrix}
   1&0\\  q&1
\end{pmatrix}.
\]
The groups $\Gamma_q$ have been studied extensively from the perspective of freeness. Brenner \cite{MR75952} proved that $\Gamma_q$ is free for $q\in \mathbb{R}\setminus(-4,4)$, and it is also known that if $q$ lies in the \textbf{Riley slice}, then $\Gamma_q$ is free \cite{MR1272421}. Moreover, $\Gamma_q$ is free whenever $q$ is transcendental \cite{MR94388}.

We consider the case $q=\tfrac{s}{t}\in\mathbb Q$ with $|q|<4$. This freeness problem was first studied by Lyndon and Ullman \cite{MR258975} and has since attracted considerable attention; see, for example, \cites{MR1245077,MR372042,MR1794285,MR4322999}. In \cite{MR1370894}, it was shown that $\Gamma_q$ is non-free when $1\leq |s|\leq 16$ and $|q|<4$. Kim and Koberda \cite{MR4505367} later extended this result to $1\leq|s|\leq 27$, excluding the case $|s|=24$. Although these works were not concerned with arithmeticity, they provide a key source of examples for the present paper. Indeed, the relations used to certify non-freeness often produce nontrivial upper-triangular elements, which are precisely the objects required by our arithmeticity criterion; see Theorem~\ref{Thm: characteristic S-arithmetic}.

Beyond the freeness problem, there has also been considerable interest in the arithmeticity of $\Gamma_q$. In \cite[Theorem~3.10]{MR5043742}, it was shown that $\Gamma_q$ is $S$-arithmetic under the Greenberg--Shalom Commensurator Hypothesis. This has been verified unconditionally in several families \cite{MR4556432,MR5004046}. The present paper gives a new criterion for arithmeticity, which in turn yields several new infinite families of congruence subgroups.

Fix $q=\tfrac{s}{t}\in\mathbb Q\setminus\{0\}$, where $\gcd(s,t)=1$ and $|q|<4$. For two numbers $a,b\in \mathbb Z\left[\tfrac{1}{t}\right]$, we write
\[a\equiv_t b\pmod s\]
if $a-b\in s\mathbb Z\left[\tfrac{1}{t}\right]$. 

Let the \textbf{congruence subgroup associated to $q=\tfrac{s}{t}$} be
\[\Gamma_1^{(t)}(s)=\left\{\begin{pmatrix}
    a&b\\ c&d
\end{pmatrix}\in\SL\left(2,\mathbb Z\left[\tfrac{1}{t}\right]\right): a\equiv_td\equiv_t1\pmod s,  \quad c\equiv_t0 \pmod s\right\}.\]
When $t=1$, we simply denote $\Gamma_1^{(1)}(s)$ by $\Gamma_1(s)$, in accordance with the usual notation for congruence subgroups.

Following \cites{MR5043742, MR5004046}, we formulate the following conjecture.
\begin{conjecture}\label{Con: arithmeticty}
    For every nonzero $q=\tfrac{s}{t}\in \mathbb Q\cap (-4,4)$, the group $\Gamma_q$ is equal to the congruence subgroup $\Gamma_1^{(t)}(s)$. 
\end{conjecture}

We provide a useful criterion for the conjecture and verify it in several cases.

Let $H=\langle A\rangle$ and let $U$ be the group of upper triangular matrices
\[
U=\left\{\begin{pmatrix}
    *&*\\ 0&*
\end{pmatrix}\in \SL(2,\mathbb{R})\right\}.
\] 
By symmetry, all statements below admit an equivalent formulation with $H$ and $U$ replaced by $\langle B_q\rangle$ and the group of lower triangular matrices; we work with $H$ for convenience. Furthermore, set $U_q=U\cap \Gamma_q$.
\begin{mthm}\label{Thm: characteristic S-arithmetic}
For $q\neq 0,\pm3$, we have $\Gamma_q=\Gamma_1^{(t)}(s)$ if and only if $U_q\neq H$. In particular, $\Gamma_q$ is finitely presented in this case.
\end{mthm}

The conclusion fails only for the exceptional parameters $q=0,\pm3$. For $q=\pm3$, the group $\Gamma_q$ is a congruence subgroup but $U_q=H$, while for $q=0$, one has $\Gamma_0=H$. 

\begin{remark}
Since \(U=\operatorname{Stab}_{\SL(2,\mathbb R)}(\infty)\), one has
\[
U_q=\operatorname{Stab}_{\Gamma_q}(\infty).
\]
Hence, Theorem~\ref{Thm: characteristic S-arithmetic} can be interpreted as saying that, for \(q\neq0,\pm3\), the group \(\Gamma_q\) is arithmetic if and only if $\operatorname{Stab}_{\Gamma_q}(\infty)$ is strictly larger than $H=\langle A\rangle$.
\end{remark}

The significance of Theorem~\ref{Thm: characteristic S-arithmetic} is that it reduces the problem of proving $\Gamma_q=\Gamma_1^{(t)}(s)$ to the considerably simpler task of producing an element of $U_q\setminus H$. At first glance, it looks difficult to find such an element. However, it turns out that many constructions used to prove that $q$ is not free naturally produce elements of $U_q\setminus H$. We illustrate this phenomenon in the following results.

We first apply Theorem~\ref{Thm: characteristic S-arithmetic} to two distinct Diophantine constructions. Although the resulting families arise from different quadratic equations, both ultimately produce the upper-triangular elements required by the criterion.

It has long been known that the non-freeness of $\Gamma_q$ is closely related to general Pell equations \cites{MR1245077,MR1420342,MR4322999}. Theorem~\ref{Thm: characteristic S-arithmetic} provides a new connection between this Diophantine perspective and arithmeticity. Applying the criterion to suitable quadratic Diophantine equations, we construct two infinite families of parameters $q$ for which $\Gamma_q$ is arithmetic. The first family arises from integral points on indefinite binary quadratic forms, while the second arises from representations of squares by indefinite quadratic forms. We also show that the families constructed in \cites{MR1245077,MR1420342} yield congruence subgroups, whereas the corresponding question remains open for the family of \cite{MR4322999}.

Define 
\[
\mathcal{CS}=\setdef{q=\tfrac{s}{t}\in \mathbb Q\cap (-4,4)}{\Gamma_q=\Gamma_1^{(t)}(s)}.
\]

We note here that although $\mathcal{CS}$ is defined using subgroups of $\SL(2,\mathbb{R})$, it is relevant to arithmeticity questions in other rank-one Lie groups. Indeed, the existence of parameters $q\in\mathcal{CS}$ often provides a sufficient condition for the arithmeticity of groups generated by opposite unipotents in more general settings. We discuss this connection in Section~\ref{sec:rank-one-general}.

\begin{mthm}\label{thm: quadratic form}
Let $A,B,C$ be three nonzero integers such that
    \begin{enumerate}
        \item $C\mid AB$
        \item $(A+B+C)^2-4AB>0$ and is not a perfect square.
    \end{enumerate}
    Let $(s,t)$ be an integral solution with $st\neq0$ of the equation 
    \[
    AB s^2+(A+B+C)st+t^2=\pm 1.
    \]
    Then $q=\tfrac{s}{t}\in \mathcal{CS}$.
\end{mthm}
Note that the equation
\[
ABs^2+(A+B+C)st+t^2=1
\]
always has an integral solution $(0,1)$. Moreover, the discriminant condition implies that the form is indefinite with a non-square discriminant. Hence, by the classical theory of binary quadratic forms, this equation admits infinitely many integral solutions. Consequently, each choice of $A,B,C$ satisfying the hypotheses of Theorem~\ref{thm: quadratic form} gives rise to infinitely many arithmetic groups $\Gamma_{q}$.

The next theorem presents a distinct infinite family of arithmetic groups.

\begin{mthm}\label{thm:Pell-equation}
Let $k\in\mathbb Z\setminus\{0,-1\}$, and let $(u,v)$ be an integral solution of
\[
(k+1)u^2-kv^2=1
\]
such that $uv\neq\pm1$. Then
\[q_\pm(u,v,k):=\tfrac{-\left(1+2(k+1)uv\right)\pm\left((k+1)u^2+kv^2\right)}{2k(k+1)uv}\in \mathcal{CS}.\]  Moreover,
\[
\lim_{u\to\infty} q_\pm(u,v,k)=\tfrac{-1}{k}\pm\sqrt{\tfrac{1}{k(k+1)}}.
\]
\end{mthm}

We conclude by exploring the computational aspects of the problem. Theorem~\ref{Thm: small denominators} is intended primarily as an illustration of Theorem~\ref{Thm: characteristic S-arithmetic}. The range $1\le |s|\le 21$ is chosen so that the verification can be recorded in a reasonably short appendix.
 
\begin{thm}\label{Thm: small denominators}
    For every $q=\tfrac{s}{t}\in \mathbb Q\cap (-4,4)$ with $1\leq |s|\leq 21$, $q\in \mathcal{CS}$.
\end{thm}
Other immediate consequences of Theorem~\ref{Thm: characteristic S-arithmetic} include the following.
\begin{corollary}\label{cor: /n}
 If $q\in\mathcal{CS}$, then $\tfrac{q}{n}\in\mathcal{CS}$ for every $n\in\mathbb Z\setminus\{0\}$.
\end{corollary}
\begin{corollary}\label{Cor: t=kspm1}
    If $t\equiv\pm1\pmod s$ and $(t^2-1)st\neq0$, then $\tfrac{s}{t}\in\mathcal{CS}$.
\end{corollary}

Before discussing the proof strategy, let us recall a common method for proving that $\Gamma_q$ is not free. Let $F(a,b)$ be the free group of rank two and define
\[
\phi_q:F(a,b)\to \SL(2,\mathbb R)
\]
by $\phi_q(a)=A$ and $\phi_q(b)=B_q$. A standard approach, used for example in \cites{MR1370894,MR4505367}, is to construct a word $w\in F(a,b)$ such that
\[
w\neq a^n,\qquad \phi_q(w)\in U.
\]

The key observation of the present paper is that many such constructions in fact produce elements of $U_q\setminus H$. By Theorem~\ref{Thm: characteristic S-arithmetic}, this stronger conclusion is sufficient to establish arithmeticity. Consequently, techniques originally developed to prove non-freeness may be reinterpreted as tools for constructing arithmetic groups.

This principle underlies all of our applications. Theorem~\ref{Thm: small denominators} follows from a careful analysis of the words constructed in \cite{MR4505367}, while Theorems~\ref{thm: quadratic form} and \ref{thm:Pell-equation} arise from different families of words that lead to quadratic Diophantine equations.

We now turn to the proof of Theorem~\ref{Thm: characteristic S-arithmetic}. The case $q\in\mathbb Z$ is straightforward, so we assume that $q\notin\mathbb Z$.

The proof of Theorem~\ref{Thm: characteristic S-arithmetic} proceeds in two stages. First, an element of $U_q\setminus H$ is bootstrapped into the unipotent lattice needed to apply the arithmeticity theorem of Venkataramana \cite{MR1306038}. This yields $S$-arithmeticity of $\Gamma_q$. Second, Serre's congruence subgroup theorem \cite{MR272790} is used to identify the resulting arithmetic subgroup with $\Gamma_1^{(t)}(s)$.

Finally, Theorem~\ref{Thm: characteristic S-arithmetic} suggests that the arithmeticity problem for $\Gamma_q$ may be governed by the structure of $U_q$. It would be interesting to determine whether non-freeness already implies $U_q\neq H$. Combined with Theorem~\ref{Thm: characteristic S-arithmetic}, such a result would establish a weak version of Conjecture~\ref{Con: arithmeticty}.

The paper is organized as follows. Section~\ref{Sec: prelinimaries} collects the preliminaries used throughout the paper. In Section~\ref{Sec: proof of Theorem A}, we prove Theorem~\ref{thm: Higher-rank} and its generalizations. Section~\ref{Sec: proof of theorem B} is devoted to the proof of the rank-one arithmeticity criterion, Theorem~\ref{Thm: characteristic S-arithmetic}. In Section~\ref{sec:rank-one-general}, we discuss an application of this criterion to groups generated by opposite unipotents in general simple real Lie groups of rank one. Theorems~\ref{THM: SL(n Z) rank gradient} and~\ref{Thm: SL(2,O_F) rank gradient} are proved in Section~\ref{Sec: 2-generated}.

The remaining sections develop applications of the rank-one criterion. Section~\ref{Sec: step 1 and 2} recalls one-step and two-step relation numbers and explains how the corresponding relations produce upper-triangular elements. In Section~\ref{Sec: quadratic forms to unipotents}, we prove Theorems~\ref{thm: quadratic form} and~\ref{thm:Pell-equation}. Theorem~\ref{Thm: small denominators} is proved in Section~\ref{Sec: proof of theorem E}. Finally, Appendix~\ref{Appen: verify} records the explicit words and matrix identities used in the proof of Theorem~\ref{Thm: small denominators}.

\section{Preliminaries}\label{Sec: prelinimaries}
In this section, we recall several results that will be used later in the paper. We state them in a slightly simplified form, which is sufficient for our applications.

\subsection{Unipotents and arithmeticity}\label{sec: horospherical lattice}

Let $F$ be a global field, and let $G$ be a connected, absolutely
almost simple algebraic group defined over $F$. Assume that
$\operatorname{rank}_F(G)\geq1$. Let $P^+$ be a
minimal $F$-parabolic subgroup of $G$, with unipotent radical $U^+$,
and let $P^-$ be an opposite minimal $F$-parabolic subgroup, with
unipotent radical $U^-$. Fix a maximal $F$-split torus
\[
T\subseteq P^+\cap P^-.
\]
We write
\[
\Phi=\Phi(G,T)
\]
for the corresponding relative root system.

Let $S$ be a finite set of places of $F$ containing all
Archimedean places when $\operatorname{char}(F)=0$, and let
\[
\mathcal O_{F,S}
=
\left\{
x\in F:
v_{\mathfrak p}(x)\geq 0
\text{ for every non-Archimedean place }
\mathfrak p\notin S
\right\}
\]
be the ring of $S$-integers. We use the notation
\[
\operatorname{rank}_{S}(G)
=
\sum_{v\in S}\operatorname{rank}_{F_v}(G).
\]

After fixing an $\mathcal O_{F,S}$-model of $G$, we write
\[
U^\pm(\mathcal O_{F,S})
\]
for the corresponding groups of $S$-integral points.

A subgroup $\Lambda\leq G(F)$ is called \textbf{$S$-arithmetic} if
$\Lambda$ and $G(\mathcal O_{F,S})$ are commensurable; equivalently, $\Lambda\cap G(\mathcal O_{F,S})$
has finite index in both $\Lambda$ and $G(\mathcal O_{F,S})$.

The following arithmeticity theorem is the form that will be used throughout the paper.

\begin{theorem}
\cite{MR1306038,MR1141802}
\label{thm: horospherical lattices}
Assume that $\operatorname{rank}_{S}(G)\geq2$. If
$\operatorname{rank}_F(G)=1$, assume in addition that
$\operatorname{char}(F)\neq2$.
Let $\Lambda^+$ and $\Lambda^-$ be finite-index subgroups of
\[
U^+(\mathcal O_{F,S})
\qquad\text{and}\qquad
U^-(\mathcal O_{F,S}),
\]
respectively. Then
\[
\langle\Lambda^+,\Lambda^-\rangle
\]
is a $S$-arithmetic subgroup of $G(F)$.
\end{theorem}

Thus, in order to prove arithmeticity, it is enough to show that the
group under consideration contains sufficiently large lattices in
both opposite unipotent radicals.

Further generalizations involving horospherical lattices can be found in \cites{MR2718935,MR4101737}.
\subsection{Congruence subgroup property}
Let $R$ be a commutative ring and let $\Gamma=\SL(n,R)$. For an ideal $I\triangleleft R$, the \textbf{principal congruence subgroup of level $I$} is defined by
\[
\Gamma(n,I)=\ker\bigl(\SL(n,R)\to \SL(n,R/I)\bigr).
\]
A subgroup $H<\Gamma$ is called a \textbf{congruence subgroup} if it contains $\Gamma(n, I)$ for some nonzero ideal $I$.

The \textbf{congruence subgroup property} asserts that every finite-index subgroup of $\Gamma$ is a congruence subgroup. This problem has a long history. In particular, Bass--Milnor--Serre solved the problem for $\mathrm{SL}(n,\mathcal O_S)$ with $n\ge 3$ \cite{MR244257}, while Serre determined the congruence kernel for $\SL(2)$ over number fields and in the $S$-arithmetic setting \cite{MR272790}.

The following consequence is sufficient for our purposes.

\begin{theorem}\cites{MR244257,MR272790}\label{Thm: CSP}
\begin{enumerate}
    \item Let $n\ge 3$, $F$ be a number field, and $\mathcal O_F$ be its ring of algebraic integers. Then $\SL(n, \mathcal O_F)$ has the congruence subgroup property if and only if $F$ is not totally imaginary.

    \item $\SL(2,\mathbb Z\left[\tfrac1N\right])$ has the congruence subgroup property when $N\geq 2$.
\end{enumerate}
\end{theorem}

\subsection{Bellman--Ford Algorithm and Difference Constraints}

We shall use the following standard consequence of the Bellman--Ford algorithm concerning systems of difference constraints.

\begin{theorem}\cite[Third Edition, Theorem 24.9]{MR2572804}\label{thm: feasuble solution}
Let $G$ be the augmented constraint graph of a system of difference constraints
\[A\vec{x}\leq \vec{b}.\] If $G$ contains no negative-weight cycle, then there is a feasible solution to the system. Moreover, if $\vec{b}$ is integral, then the feasible solution is also integral.
\end{theorem}
\subsection{Net subgroups and net determinants}
We briefly recall some results on commutative rings from \cite{MR669558}.

Let $R$ be a commutative semilocal ring with identity. Suppose we are given a square array
\[
\sigma=(\sigma_{ij}),\qquad 1\le i,j\le n,
\]
whose entries are ideals of $R$. We call $\sigma$ a \textbf{net of ideals} of order $n$ if
\[
\sigma_{ik}\sigma_{kj}\subseteq \sigma_{ij}
\]
for all indices $i,j,k$. Throughout this subsection, we assume that
\[
\sigma_{ii}=R,\qquad 1\le i\le n.
\]

For any subset $I\subseteq\{1,2,\ldots,n\}$, define
\[
\sigma_I=\sum_{i\in I,\;j\notin I}\sigma_{ij}\sigma_{ji}.
\]
For a matrix $A\in\GL(n,R)$, let $\operatorname{det}_I(A)$ denote the determinant of the principal submatrix of $A$ whose rows and columns are indexed by $I$.

We are interested in the subgroup
\[
\Gamma(\sigma)
=
\left\langle
E_{ij}(q_{ij})
:
q_{ij}\in\sigma_{ij},\ i\neq j
\right\rangle.
\]

The following characterization is proved in \cite{MR669558, MR717569}.

\begin{theorem}\label{thm: congruence conditions}
Let $R$ be a commutative semilocal Bézout ring with identity, and let $\sigma$ be a net of ideals of order $n$. Then a matrix $A\in\GL(n,R)$ belongs to $\Gamma(\sigma)$ if and only if the following conditions are satisfied:
\begin{enumerate}
    \item For all $i\neq j$,
    $A_{ij}\in\sigma_{ij}$.
    \item For every subset $I\subseteq\{1,2,\ldots,n\}$, $\operatorname{det}_I(A)-1\in\sigma_I$.
\end{enumerate}
\end{theorem}

In \cite{MR687835}, it was shown that it is sufficient to verify the determinant condition for at most $n$ subsets. Generalizations of this result to other types of Chevalley groups can be found in the survey \cite{MR687837}.

\section{Proof of Theorem~\ref{thm: Higher-rank}}\label{Sec: proof of Theorem A}
In this section, we prove Theorem~\ref{thm: Higher-rank} and also determine the associated $S$-arithmetic group $\Gamma(Q)$. The proof is based on a theorem of Tits together with a criterion of Venkataramana and its generalizations. As a consequence, the determination of $\Gamma(Q)$ is reduced to a finite congruence computation. This computation can be carried out explicitly, yielding a complete description of the resulting groups. 

The proof proceeds in four steps. First, we modify the parameter matrix by a sequence of elementary operations until its local valuations satisfy a suitable admissibility condition. Second, admissibility allows us to construct, via a system of difference constraints, a diagonal conjugating matrix that places the generators inside an appropriate ring of $S$-integers. Third, combining this construction with an idea of Venkataramana yields the $S$-arithmeticity of the resulting group. Finally, the congruence subgroup property reduces the determination of the group to a finite computation over a finite quotient, where the theory of net subgroups applies.

We also discuss $S$-arithmeticity for general higher-rank algebraic groups. Since the proof is essentially identical, we present the argument first for $\SL(n)$, where the ideas are particularly transparent. 
\subsection{Admissible reduction} 
As is often the case in group theory, understanding a group requires a suitable generating set. In this subsection, we modify the given generating set so that it satisfies additional multiplicative relations. We will work with general number fields.

We begin with some notation. Let $F$ be a number field, and let $\mathcal O_F$ be its ring of algebraic integers. Every nonzero fractional ideal $\mathfrak a$ admits a unique factorization
\[
\mathfrak a=\prod_{i=1}^l \mathfrak p_i^{k_i},
\]
where the $\mathfrak p_i$ are distinct nonzero prime ideals of $\mathcal O_F$ and $k_i\in \mathbb Z$. We denote by
\[
\Pi(\mathfrak a)=\{\mathfrak p_i\}_{i=1}^l
\]
the set of prime ideals appearing in this factorization.

For each nonzero prime ideal $\mathfrak p$, let
\[
v_{\mathfrak p}:I_F\longrightarrow\mathbb Z
\]
denote the corresponding valuation, where $I_F$ is the group of nonzero fractional ideals. Thus, $v_{\mathfrak p}(\mathfrak a)=k$
if and only if $\mathfrak p^k$ is the $\mathfrak p$-part of
$\mathfrak a$. For any finite set $S$ of nonzero prime ideals, define 
\[
\mathcal O_{F,S}=\setdef{x\in F} {v_{\mathfrak p}(x)\ge 0
\text{ for all nonzero prime ideals } \mathfrak{p}\notin S}.
\]

\begin{defn}
A \textbf{parameter matrix over $F$} is a collection
\[
\mathfrak{Q}
=(\mathfrak q_{ij})_{1\le i,j\le n,\ i\neq j}\]
of nonzero fractional ideals of $F$ indexed by ordered pairs $(i,j)$ with $i\neq j$.

For convenience, we regard $\mathfrak{Q}$ as an $n\times n$ matrix whose diagonal entries are left undefined.
\end{defn}
\begin{defn}
    Let $\mathfrak{Q}$ be a parameter matrix over $F$. The \textbf{group generated by elementary elements of type $\mathfrak{Q}$} is defined by 
    \[
    \Gamma(\mathfrak{Q}):=\langle E_{ij}(q_{ij}): q_{ij}\in \mathfrak{q}_{ij}\rangle.
    \]
\end{defn}

For any nonzero prime ideal $\mathfrak{p}$ and parameter matrix $\mathfrak{Q}=(\mathfrak{q}_{ij})$, we define
\[
v_{\mathfrak{p}}(\mathfrak{Q}) := (v_{\mathfrak{p}}(\mathfrak{q}_{ij}))_{i\neq j},
\]
the off-diagonal $\mathfrak{p}$-adic valuation matrix of $\mathfrak{Q}$.

The purpose of the following definition is to simplify the detection of negative cycles, which will be defined later in Section~\ref{sec: find N and g}. Without additional assumptions, a negative cycle may involve arbitrarily many vertices. Under the admissibility condition, however, every negative cycle contains a negative $2$-cycle. This reduction will be established in Lemma~\ref{lem: Pi_1} and is the key ingredient in the construction of the diagonal conjugating matrix.

Let $Z=(z_{ij})$ be an off-diagonal integral $n\times n$ matrix.  
We say that $Z$ is \textbf{admissible} if either $Z \le 0$, or
\[
z_{ij} \le z_{ik} + z_{kj} \quad \text{for all } i \ne j,\ k \ne i,j.
\]

The following elementary operation does not change the generated subgroup. It is the ideal-theoretic analogue of replacing an elementary generator by a commutator relation, and will be used repeatedly to simplify the parameter matrix while preserving the associated group.

\begin{defn}
Let $\mathfrak{Q}$ be an off-diagonal $n\times n$ parameter matrix over $F$. We write
$\mathfrak{Q} \rightsquigarrow \mathfrak{R}$
if $\mathfrak{R}$ is obtained from $\mathfrak{Q}$ by replacing a single entry $\mathfrak{q}_{ij}$
with $\mathfrak{q}_{ij}+\mathfrak{q}_{ik}\mathfrak{q}_{kj}$ for some $k\neq i,j$.
We denote by $\mathfrak{Q} \rightsquigarrow^*\mathfrak{R}$ the reflexive-transitive closure of this relation.
\end{defn}

The following result is a consequence of the definitions.

\begin{lemma}\label{lem: p one by one}
If $\mathfrak{Q_1} \rightsquigarrow \mathfrak{Q_2}$, then
$\Gamma(\mathfrak{Q_1})=\Gamma(\mathfrak{Q_2})$.
Furthermore, if $v_{\mathfrak{p}}(\mathfrak{Q_1})$ is admissible for some nonzero prime ideal $\mathfrak{p}$,
then so is $v_{\mathfrak{p}}(\mathfrak{Q_2})$.
\end{lemma}

\begin{proposition}\label{Pro: admissible}
For every parameter matrix $\mathfrak{Q}$, there exists a parameter matrix $\mathfrak{Q'}$ such that
\[
\mathfrak{Q} \rightsquigarrow^* \mathfrak{Q'},
\]
and $v_{\mathfrak{p}}(\mathfrak{Q'})$ is admissible for all nonzero prime ideals $\mathfrak{p}$.
\end{proposition}

We first treat a special case.

\begin{lemma}\label{lem: admission for one q}
For every parameter matrix $\mathfrak{Q}$ and every nonzero prime ideal $\mathfrak{p}$, there exists a parameter matrix $\mathfrak{Q'}$ such that
\[
\mathfrak{Q} \rightsquigarrow^* \mathfrak{Q'},
\]
and $v_{\mathfrak{p}}(\mathfrak{Q'})$ is admissible.
\end{lemma}

\begin{proof}
Let us first assume $n\ge 3$.

If $v_{\mathfrak{p}}(\mathfrak{Q})$ is admissible, we are done by taking $\mathfrak{Q'}=\mathfrak{Q}$.

Otherwise, $v_{\mathfrak{p}}(\mathfrak{Q})$ is not admissible. 
Hence there exist pairwise distinct indices $i,j,k$ such that
\[
v_{\mathfrak{p}}(\mathfrak{q}_{ij}) > v_{\mathfrak{p}}(
\mathfrak{q}_{ik}) + v_{\mathfrak{p}}(
\mathfrak{q}_{kj}).
\]

We replace the entry $\mathfrak q_{ij}$ by
$\mathfrak q_{ij}+\mathfrak q_{ik}\mathfrak q_{kj}$, obtaining a new matrix $\mathfrak{Q_1}$ such that
\[
\mathfrak{Q} \rightsquigarrow \mathfrak{Q_1},
\quad
v_{\mathfrak{p}}(\mathfrak{Q_1})_{ij} < v_{\mathfrak{p}}(\mathfrak{Q})_{ij}.
\]

We repeat this process while $v_{\mathfrak{p}}(\mathfrak{Q})$ is not admissible.

At each step, at least one entry of the $\mathfrak{p}$-adic valuation matrix strictly decreases. Let
\[
M=\max_{i\neq j}|v_{\mathfrak{p}}(\mathfrak{q}_{ij})|.
\]
If the process has not terminated after $3n^2M$ steps, then some entry $v_{\mathfrak{p}}(\mathfrak{q}_{ij})$ must have decreased by more than $3M$. Since initially all entries lie in $[-M,M]$, it follows that in the resulting matrix $\mathfrak{Q}_2=(\mathfrak{q'_{ij}})$, 
\[
v_{\mathfrak{p}}(\mathfrak{q}'_{ij})\le -2M.
\]

Fix such a pair $(i,j)$. For any $k\neq i,j$, perform the replacement
through the path $i\to j\to k$. This replaces
$v_{\mathfrak p}(\mathfrak q'_{ik})$ by
\[
\min\{v_{\mathfrak{p}}(\mathfrak{q}'_{ik}),v_{\mathfrak{p}}(\mathfrak{q}'_{ij})+v_{\mathfrak{p}}(\mathfrak{q}'_{jk})\}.
\]
Since $v_{\mathfrak{p}}(\mathfrak{q}'_{ij})\le -2M$ and $v_{\mathfrak{p}}(\mathfrak{q}'_{jk})\le M$, the new value of $v_{\mathfrak{p}}(\mathfrak{q}'_{ik})$ is at most $-M$. Similarly, using the replacement through the path $k\to i\to j$, we can make
\[
v_{\mathfrak{p}}(\mathfrak{q}'_{kj})\le -M.
\]
Repeating this for all $k$ shows that every entry in the $i$-th row and the $j$-th column has $\mathfrak{p}$-adic valuation at most $-M$.

The remaining cases $k=j$ or $l=i$ are handled similarly by using the negative entries in the $i$-th row and $j$-th column and performing at most one further replacement through the path $i\to k\to l$ for some $k\neq i,j$. Thus, after finitely many additional replacements, every off-diagonal entry of the $\mathfrak{p}$-adic valuation matrix is non-positive. Hence
\[
v_{\mathfrak{p}}(\mathfrak Q')\le 0,
\]
and is therefore admissible.

Thus, after at most $3n^2M+O(n^2)$ steps, the process produces an admissible matrix.

Note that when $n=2$, every integral matrix is admissible. The result follows.
\end{proof}

\begin{proof}[Proof of Proposition~\ref{Pro: admissible}]
Let
\[
\Pi'_3(\mathfrak{Q})=\{\mathfrak{p} : v_{\mathfrak{p}}(\mathfrak Q)=0\}.
\]
Then there are only finitely many nonzero prime ideals outside $\Pi'_3(\mathfrak{Q})$; enumerate them as $\mathfrak{p}_1,\dots,\mathfrak{p}_s$.

Apply Lemma~\ref{lem: admission for one q} successively to $\mathfrak{p}_1,\dots,\mathfrak{p}_s$. Each step preserves admissibility at the previously treated prime ideals by Lemma~\ref{lem: p one by one}. The result follows.
\end{proof}
\begin{defn}
     We call a parameter matrix $\mathfrak{Q}$ \textbf{$*$-admissible} if its \(\mathfrak p\)-adic valuation matrix is admissible for all nonzero prime ideals $\mathfrak{p}$.
\end{defn}

By Proposition~\ref{Pro: admissible}, we may assume from now on that all parameter matrices $\mathfrak{Q}$ are $*$-admissible.
\subsection{The number \texorpdfstring{$N$ and the matrix $g$}{N and the matrix g}}\label{sec: find N and g}

In this subsection, we define the number $N$ and the matrix $g$ in Theorem~\ref{thm: Higher-rank} under the assumption that $\mathcal O_F$ is a PID.

Let $\mathfrak{Q}=(\mathfrak{q}_{ij})$ be a parameter matrix, and let $\mathfrak{p}$ be a nonzero prime ideal. We say that $\mathfrak{p}$ admits a \textbf{negative cycle over $\mathfrak{Q}$} if there exists a cycle
\[
i_1 \to i_2 \to \cdots \to i_k \to i_{k+1}=i_1, \quad i_t\neq i_{t+1}
\]
such that
\[
\sum_{j=1}^k v_{\mathfrak{p}}(\mathfrak{q}_{i_ji_{j+1}})<0.\]
For a $*$-admissible parameter matrix $\mathfrak{Q}$, the prime ideals admitting a negative cycle are easy to characterize.
\begin{lemma}\label{lem: Pi_1}
  Let $\mathfrak Q=(\mathfrak q_{ij})$ be a $*$-admissible parameter
matrix. A nonzero prime ideal $\mathfrak p$ admits a negative cycle
over $\mathfrak Q$ if and only if there exist $i\neq j$ such that
  \[
        v_{\mathfrak{p}}(\mathfrak{q}_{ij})+v_{\mathfrak{p}}(\mathfrak{q}_{ji})<0.
  \]
\end{lemma}
\begin{proof}
   Let $\mathfrak{p}$ be a nonzero prime ideal that admits a negative cycle over $\mathfrak{Q}$. Let 
   \[
i_1 \to i_2 \to \cdots \to i_k \to i_{k+1}=i_1, \quad i_t\neq i_{t+1}
\]
be a negative cycle. 

If $v_\mathfrak{p}(\mathfrak{Q})\leq 0$, then at least one of the values $v_{\mathfrak{p}}(\mathfrak{q}_{i_ji_{j+1}})$, $1\le j\leq k$ is negative, say the value corresponding to $j=1$. Hence 
\[v_{\mathfrak{p}}(\mathfrak{q}_{i_1i_2})+v_{\mathfrak{p}}(\mathfrak{q}_{i_2i_1})<0.\]

Otherwise, since $\mathfrak{Q}$ is $*$-admissible, we have, for all pairwise distinct indices $i,j,k$,
\[
v_\mathfrak{p}(\mathfrak{q}_{ij})\leq v_\mathfrak{p}(\mathfrak{q}_{ik})+v_\mathfrak{p}(\mathfrak{q}_{kj}).
\]
Let $1=n_1<n_2<\cdots< n_l=k+1$ be the indices such that $i_{n_\alpha}=i_1$. Then by definition, one of the cycles from $i_{n_\alpha}$ to $i_{n_{\alpha+1}}$ is a negative cycle in which none of the intermediate vertices is $i_1$. Hence, without loss of generality, we may assume $i_j\neq i_1$ for all $2\le j\le k$.

Then, 
\[
v_{\mathfrak{p}}(\mathfrak{q}_{i_1i_k})\leq \sum_{j=1}^{k-1}v_{\mathfrak{p}}(\mathfrak{q}_{i_ji_{j+1}}).
\]
It follows that
\[
v_{\mathfrak{p}}(\mathfrak{q}_{i_1i_{k}})+v_{\mathfrak{p}}(\mathfrak{q}_{i_ki_1})<0.
\]

Conversely, if
\[
v_{\mathfrak{p}}(\mathfrak{q}_{ij})+v_{\mathfrak{p}}(\mathfrak{q}_{ji})<0,
\]
then the 2-cycle $i\to j\to i$ is a negative cycle.

This completes the proof.
\end{proof}

Let $\mathfrak Q$ be a $*$-admissible parameter matrix. Denote by $\Pi'_1(\mathfrak Q)$ the set of nonzero prime ideals admitting a negative cycle over $\mathfrak Q$, and let $\Pi_1(\mathfrak Q)$ be the union of $\Pi'_1(\mathfrak Q)$ and the set of Archimedean places of $F$.

Now we construct $g$ as stated in Theorem~\ref{thm: Higher-rank}. Theorem~\ref{thm: feasuble solution} will be used in this construction. We begin by translating the problem into a system of difference constraints.

\begin{lemma}\label{lemma: find g}
Let $\mathfrak{Q}$ and $\Pi'_1(\mathfrak{Q})$ be as defined above. For any nonzero prime ideal $\mathfrak{p}\notin \Pi'_1(\mathfrak{Q})$, there exists a diagonal matrix $g_\mathfrak{p}\in \GL(n,F)$ such that $v_\mathfrak{p}(g_\mathfrak{p}\mathfrak{Q}g_{\mathfrak{p}}^{-1})\geq 0$.

Furthermore, if $\mathcal O_F$ is a PID, there exists a diagonal matrix $g\in \GL(n,F)$ such that
\[
g\Gamma(\mathfrak{Q})g^{-1}\le \SL(n,\mathcal O_{F,\Pi_1(\mathfrak{Q})}).
\]
\end{lemma}

\begin{proof}
Let $\mathfrak{p}\notin \Pi'_1(\mathfrak{Q})\cup \Pi'_3(\mathfrak{Q})$ be a nonzero prime ideal. If $v_{\mathfrak{p}}(\mathfrak{Q})\ge 0$, we set $g_{\mathfrak{p}}=I_n$.

Otherwise, $v_{\mathfrak{p}}(\mathfrak{Q})$ is admissible. Since $\mathfrak{p}\notin \Pi'_1(\mathfrak{Q})$, $\mathfrak{p}$ does not admit any negative cycles. Fix an element $p\in\mathcal O_F$ such that $v_\mathfrak{p}(p)=1$. We claim that for a suitable choice of
\[
g_\mathfrak{p}=\operatorname{diag}(p^{a_1},p^{a_2},\ldots,p^{a_n}),
\]
where $a_1,\ldots,a_n$ are integers, we have
\begin{equation}\label{eq: v_q requirements}
v_{\mathfrak{p}}(g_\mathfrak{p}\mathfrak{Q}g_\mathfrak{p}^{-1})\ge 0.
\end{equation}

Observe that equation~\eqref{eq: v_q requirements} is equivalent to requiring the following inequalities for all $i\neq j$:
\begin{equation}\label{eq: kill p}
a_i-a_j\le v_{\mathfrak{p}}(\mathfrak{q}_{ji}), \qquad a_j-a_i\le v_{\mathfrak{p}}(\mathfrak{q}_{ij}).
\end{equation}
This is a system of difference constraints. The associated graph is the complete directed graph $DG$ on $n$ vertices, whose edge weights are given by
\[
w(i,j)=v_{\mathfrak{p}}(\mathfrak{q}_{ij}).
\]

By Theorem~\ref{thm: feasuble solution}, in order to obtain an integral solution, it suffices to show that $DG$ contains no negative-weight cycles. This is equivalent to the condition $\mathfrak{p}\notin \Pi'_1(\mathfrak{Q})$. Therefore, the system admits an integral solution $(a_1,\ldots,a_n)$, and $g_\mathfrak{p}$ is defined accordingly.

If, furthermore, $\mathcal O_F$ is a PID, we choose $p$ to be a generator of the ideal $\mathfrak{p}$ in the construction above, and set
\[
g=\prod_{\mathfrak{p}\notin \Pi'_1(\mathfrak{Q})\cup\Pi'_3(\mathfrak{Q})} g_\mathfrak{p}.
\]
Then
\[
g\Gamma(\mathfrak{Q})g^{-1}=\Gamma(g\mathfrak{Q}g^{-1})\le \SL(n,\mathcal O_{F,\Pi_1(\mathfrak{Q})}).
\]
\end{proof}

\subsection{Proof of Theorem~\ref{thm: Higher-rank}}\label{sec: classification}

In this subsection, we prove Theorem~\ref{thm: high rank number fields}, which is a generalization of Theorem~\ref{thm: Higher-rank}. We first state the result.
\begin{theorem}\label{thm: high rank number fields}
    Let $n\geq 3$, $F$ be a number field, and $\mathcal O_{F}$ be the ring of algebraic integers in $F$. For every parameter matrix $\mathfrak{Q}$, $\Gamma(\mathfrak{Q})$ is $\Pi_1(\mathfrak{Q})$-arithmetic.

    Furthermore, if $\mathcal O_F$ is a PID, then there exists a diagonal matrix $g\in \GL(n,F)$ such that $g\Gamma(\mathfrak{Q})g^{-1}$ is a finite-index subgroup of $\SL(n,\mathcal O_{F,\Pi_1(\mathfrak{Q})})$.
\end{theorem}

The main technical result is the following lemma. Under slightly stronger assumptions, it is due to Venkataramana; we use the version stated in \cite{MR4812042}.
\begin{lemma}\cite[Lemma 4.1]{MR4812042}\label{lem: small lattice to big one}
    Let $G$ be a connected, almost $F$-simple algebraic group defined
over a global field $F$. Let $T$ be a finite set of places of $F$ (containing all
Archimedean places if $\mathrm{char}(F)=0)$, and let $\Lambda$ be $T$-arithmetic.

Let $\Theta\le G(F)$ be a subgroup containing $\Lambda$ whose projection to $G(F_s)$ is bounded for almost all places $s$, and let $S$ be the (finite) set of places where the projection of $\Theta$ is unbounded. Then $\Theta$ is $S$-arithmetic.
\end{lemma}
To apply the lemma, we first prove the following result.

 \begin{lemma}\label{lem: unboundedness}
    Let $\mathfrak Q=(\mathfrak q_{ij})$ be a $*$-admissible parameter matrix. Then $\mathfrak p\in\Pi'_1(\mathfrak Q)$ if and only if $\Gamma(\mathfrak Q)$ is unbounded in $\SL(n,F_{\mathfrak p})$.
\end{lemma}
    \begin{proof}
    Let $\mathfrak{p}\in \Pi'_1(\mathfrak Q)$. By Lemma~\ref{lem: Pi_1}, there exist $i\neq j$ such that 
    \[
    v_{\mathfrak{p}}(\mathfrak{q}_{ij})+v_{\mathfrak{p}}(\mathfrak{q}_{ji})<0.
    \]
   For suitable $q_{ij}\in\mathfrak q_{ij}$ and $q_{ji}\in\mathfrak q_{ji}$, the matrix $E_{ij}(q_{ij})E_{ji}(q_{ji})$ has trace $n+q_{ij}q_{ji}$ with $v_{\mathfrak p}(n+q_{ij}q_{ji})<0$. Hence $\Gamma(\mathfrak Q)$ is unbounded in $\SL(n,F_{\mathfrak p})$.

    Let $\mathfrak{p}\notin \Pi'_1(\mathfrak Q)$. By Lemma~\ref{lemma: find g}, there is $g_\mathfrak{p}\in \GL(n,F)$ such that $v_\mathfrak{p}(g_\mathfrak{p}\mathfrak{Q}g_{\mathfrak{p}}^{-1})\geq 0$. It follows that $\Gamma(g_\mathfrak{p}\mathfrak{Q}g_{\mathfrak{p}}^{-1})$ is bounded in $\SL(n,F_\mathfrak{p})$. Hence $\Gamma(\mathfrak{Q})$ is bounded in $\SL(n,F_\mathfrak{p})$.
    \end{proof}

\begin{proof}[Proof of Theorem~\ref{thm: high rank number fields}]
By Proposition~\ref{Pro: admissible}, we may assume that $\mathfrak{Q}$ is admissible. Note that if $\mathfrak{Q}\rightsquigarrow \mathfrak{Q'}$, then $\Pi_1(\mathfrak{Q})=\Pi_1(\mathfrak{Q'})$. Hence, this reduction does not change the set $\Pi_1(\mathfrak{Q})$.

Let $\mathfrak S=(\mathfrak s_{ij})$ be the parameter matrix defined by
\[
\mathfrak s_{ij}=\mathfrak q_{ij}\cap\mathcal O_F
\qquad\text{for all }i\neq j.
\] 
We have
$\Gamma(\mathfrak{S})
\le
\Gamma(\mathfrak{Q})$.
By a theorem of Tits \cite{MR424966},
$\Gamma(\mathfrak{S})$
has finite index in
$\SL(n,\mathcal O_F)$.
In particular, it is an arithmetic lattice in the product of the groups $\SL(n,F_v)$ over all Archimedean places $v$. Hence, $\Gamma(\mathfrak{Q})$ is unbounded over all Archimedean places.

We now apply Lemma~\ref{lem: small lattice to big one} with
\[
G=\SL(n), \qquad
k=F, \quad
\Lambda=\Gamma(\mathfrak{S}).
\]
The hypotheses of Lemma~\ref{lem: small lattice to big one} are satisfied with
\(\Theta=\Gamma(\mathfrak{Q})\).
By Lemma~\ref{lem: unboundedness}, $\Gamma(\mathfrak{Q})$ is $\Pi_1(\mathfrak{Q})$-arithmetic.

Now, assume that $\mathcal O_F$ is a PID. Lemma~\ref{lemma: find g}
provides a diagonal matrix $g\in \GL(n,F)$ such that $\mathfrak{Q_1}:=g\mathfrak{Q}g^{-1}\in M(n,\mathcal O_{F,\Pi_1(\mathfrak{Q})})$. Note that $\mathfrak{Q_1}$ is admissible and $\Pi_1(\mathfrak{Q})=\Pi_1(\mathfrak{Q_1})$.

Since $\Gamma(\mathfrak{Q_1})\le \SL(n,\mathcal O_{F,\Pi_1(\mathfrak{Q})})$ and is $\Pi_1(\mathfrak{Q})$-arithmetic,
$\Gamma(\mathfrak{Q_1})$ is a finite-index subgroup of $\SL(n,\mathcal O_{F,\Pi_1(\mathfrak{Q})})$.
\end{proof}

The precise determination of the associated $S$-arithmetic group will be carried out in the next subsection.

\subsection{The group structure}
The main goal of this subsection is to prove the following classification theorem. The result extends Theorem~\ref{thm: high rank number fields}. By the congruence subgroup property, the classification reduces to a computation in a finite group. This type of problem has been extensively studied; for example, the results in \cites{MR669558, MR687835} are sufficient for us.

The explicit congruence description requires the existence of $g\in\GL(n,F)$ as in Lemma~\ref{lemma: find g}. Accordingly, we only state it for PIDs.

\begin{thm}\label{thm: classification}
Let $n\ge3$, $F$ be a number field, and $\mathcal O_F$ be the ring of algebraic integers in $F$. Assume that $F$ is not totally imaginary and $\mathcal O_F$ is a PID. For every $*$-admissible parameter matrix $\mathfrak Q$, there exist an explicitly computable matrix $g\in\GL(n,F)$ and explicitly computable pairs $(f_\alpha,m_\alpha)$, where each $f_\alpha$ is either an off-diagonal coordinate function or the determinant of a principal minor minus $1$, such that
\[
g\Gamma(\mathfrak Q)g^{-1}
=
\{
A\in \SL(n,\mathcal O_{F,\Pi_1(\mathfrak{Q})})
:
f_\alpha(A)\in m_\alpha\mathcal O_{F,\Pi_1(\mathfrak{Q})}
\text{ for all }\alpha
\}.
\]
\end{thm}
\begin{proof}
  Let $\mathfrak Q$ be a $*$-admissible parameter matrix. By Theorem~\ref{thm: high rank number fields}, there exists $g\in\GL(n,F)$ such that $g\Gamma(\mathfrak Q)g^{-1}$ is a finite-index subgroup of $\SL(n,\mathcal O_{F,\Pi_1(\mathfrak Q)})$. By Theorem~\ref{Thm: CSP}, this ambient group has the congruence subgroup property. Hence there exists an ideal $\mathfrak m$ such that
\[
\Gamma(n,\mathfrak m)\leq g\Gamma(\mathfrak Q)g^{-1}.
\]

Since $\mathcal O_F$ is a PID, we may identify $\mathfrak{Q}$ with an off-diagonal matrix $Q=(q_{ij})$ in which $q_{ij}$ is a generator of the fractional ideal $\mathfrak{q}_{ij}$. 
The prime ideals occurring in the construction play three different roles. The primes in $\Pi'_1(Q)$ determine the ambient ring of $S$-integers, the primes in $\Pi_2(Q)$ contribute only to the congruence conditions, while the remaining primes play no further role. Accordingly, we partition the nonzero prime ideals into the following three classes:
\begin{itemize}
    \item $\Pi'_1(Q)$.
    \item $\Pi_2(Q):=\setdef{\mathfrak{p}}{v_\mathfrak{p}(gQg^{-1})_{ij}>0\ \text{for some $i\neq j$}}$.
    \item $\Pi_3(Q):=\setdef{\mathfrak{p}}{v_\mathfrak{p}(gQg^{-1})=0}$.
\end{itemize}
For $\mathfrak{p}\in \Pi_2(Q)$, let \[a_\mathfrak{p}=\max_{i\neq j}\{v_\mathfrak{p}(gQg^{-1})_{ij}\}.\]
We replace $\mathfrak{m}$ with 
\[
\mathfrak{m}\prod_{\mathfrak{p}\in \Pi_2(Q)}\mathfrak{p}^{2a_\mathfrak{p}}.\]
Replacing $\mathfrak m$ by this smaller ideal ensures that reduction
modulo $\mathfrak m$ retains all divisibility information contributed
by the primes in $\Pi_2(Q)$. In particular, the congruence conditions
determined by the conjugated parameter matrix remain visible in the
finite quotient.

Let
\[
R=\mathcal O_{F,\Pi_1(\mathfrak Q)}.
\]
Let $\sigma=(\sigma_{ij})$ be the net of ideals in $R$ induced by the conjugated parameter matrix, with
\[
\sigma_{ii}=R.
\]
For each subset $I\subseteq\{1,\ldots,n\}$, define
\[
\sigma_I
=
\sum_{i\in I,\;j\notin I}\sigma_{ij}\sigma_{ji}.
\]

By our choice of $\mathfrak m$,
\[
\mathfrak m\subseteq\sigma_{ij}
\qquad\text{and}\qquad
\mathfrak m\subseteq\sigma_I
\]
for all $i,j$ and all subsets $I\subseteq\{1,\ldots,n\}$.

Let
\[
\overline R=R/\mathfrak m
\]
and define a net of ideals
\[
\overline{\sigma}
=
(\overline{\sigma}_{ij})
\]
in $\overline R$ by
\[
\overline{\sigma}_{ij}
=
\sigma_{ij}/\mathfrak m,
\qquad
1\leq i,j\leq n.
\]
For each subset $I\subseteq\{1,\ldots,n\}$, we then have
\[
\overline{\sigma}_I
=
\sigma_I/\mathfrak m.
\]

Since
\[
\Gamma(n,\mathfrak m)
\leq
g\Gamma(\mathfrak Q)g^{-1},
\]
the subgroup $g\Gamma(\mathfrak Q)g^{-1}$ is the full preimage of its
image under the reduction map
\[
p_{\mathfrak m}:\SL(n,R)\longrightarrow \SL(n,\overline R).
\]
Moreover,
\[
p_{\mathfrak m}\bigl(g\Gamma(\mathfrak Q)g^{-1}\bigr)=\Gamma(\overline{\sigma}).
\]

Because $\overline R$ is a commutative semilocal B\'ezout ring with identity, Theorem~\ref{thm: congruence conditions} implies that this image is characterized by the conditions
\[
\overline A_{ij}\in\overline{\sigma}_{ij}\qquad (i\neq j)
\]
and
\[
\operatorname{det}_I(\overline A)-1\in \overline{\sigma}_I
\]
for every subset $I\subseteq\{1,\ldots,n\}$.

The correspondence between ideals of $R$ containing $\mathfrak m$ and ideals of $R/\mathfrak m$ shows that these conditions lift exactly to
\[
A_{ij}\in\sigma_{ij}
\]
and
\[
\operatorname{det}_I(A)-1\in\sigma_I.
\]
Therefore, these finitely many congruence conditions determine $g\Gamma(\mathfrak Q)g^{-1}$.

This completes the proof.
\end{proof}

The construction in the proof of Theorem~\ref{thm: classification} is efficient. Let us clarify this by establishing an algorithm for a rational parameter matrix.

Let $Q$ be a rational parameter matrix of size $n\geq 3$.
\begin{enumerate}
    \item By Proposition~\ref{Pro: admissible}, replace $Q$ with a $*$-admissible $Q_1$. Let $\Pi'_1$ denote the set of primes admitting negative cycles over $Q_1$, and define
    \[N=\prod_{p\in\Pi'_1} p.\]
    \item Lemma~\ref{lemma: find g} provides a matrix $g\in \GL(n,\mathbb Q)$ such that $v_{p}(gQ_1g^{-1})\ge 0$ for every $p\notin \Pi_1$.  Let $\Pi_2$ be the set of primes for which $v_{p}(gQ_1g^{-1})\neq 0$. Define a matrix $Q_2=(q'_{ij})$ by 
    \[
    q'_{ij}=\prod_{p\in\Pi_2} p^{v_p((gQ_1g^{-1})_{ij})}
    \]
    \item For every proper nonempty subset $I\subset\{1,2,\cdots,n\}$, define 
    \[q'_I=\gcd\{q'_{ij}q'_{ji}:i\in I, j\notin I\}.\]
\end{enumerate}
Then  \[
g\Gamma(Q)g^{-1}=\setdef{A\in \SL\left(n,\mathbb Z\left[\tfrac{1}{N}\right]\right)}{q'_{ij}\mid A_{ij}, \quad \text{$q'_I\mid \operatorname{det}_I(A)-1$}}.
\]

\begin{example}
    Let $Q$ be the $n\times n$ matrix whose off-diagonal entries are all equal to $\tfrac{s}{t}$. If $n\ge 3$, then 
    \[
    \Gamma(Q)=\setdef{A\in \SL\left(n,\mathbb Z\left[\tfrac{1}{t}\right]\right)}{s\mid A_{ij}, \quad \text{$s^2\mid A_{ii}-1$}}.
    \]
   This may be viewed as a generalization of a result of Mennicke \cite{MR1754994}, but follows from a completely different argument.
\end{example}

\subsection{General higher-rank groups}\label{Sec: general case}

The proof of Theorems~\ref{thm: high rank number fields} and~\ref{thm: classification} uses four main ingredients: Tits's theorem, the diagonal conjugation argument, the congruence subgroup property, and the structure of net subgroups. Consequently, the same strategy can be applied to other higher-rank groups whenever suitable analogues of these ingredients are available.

In this subsection, we briefly indicate how the arithmeticity part of the argument extends to a more general setting. We do not attempt to give a complete treatment, since the effective congruence classification is substantially more delicate outside $\SL(n)$.

Tits's theorem has been generalized to many higher-rank groups by the work of Vaserstein \cite{MR349864}, Raghunathan \cite{MR1141802}, and Venkataramana \cite{MR1306038}. See Theorem~\ref{thm: horospherical lattices}.

Let $G$ be an absolutely simple algebraic group defined over a number
field $F$, and use the notation $P^\pm$, $U^\pm$, $T$, and
$\Phi=\Phi(G,T)$ introduced in
Section~\ref{sec: horospherical lattice}. Let $\Phi^+$ be the set of
positive relative roots determined by $P^+$.

For each relative root $\beta$, choose an $F$-basis
\[
X_{\beta,1},\ldots,X_{\beta,d_\beta}
\]
of $\mathfrak g_\beta$. After multiplying these basis vectors by a sufficiently divisible element of $\mathcal O_F$, we may assume that all structure constants occurring in brackets between the root spaces belong to $\mathcal O_F$, and, moreover, that the polynomial commutator relations among the one-parameter subgroups
\[
u_{\beta,r}(x)=\exp(xX_{\beta,r})
\]
have coefficients in $\mathcal O_F$. We fix such a choice throughout and refer to it as a system of good root coordinates.

Define
\[
u_{\beta,r}:F\longrightarrow U_\beta,
\qquad
1\leq r\leq d_\beta,
\]
by 
\[
u_{\beta,r}(t)=\exp(tX_{\beta,r}),
\]
where $d_\beta=\dim_F\mathfrak g_\beta$, and set
\[
\mathcal R_{\mathrm{hr}}=\setdef{(\beta,r)}{\beta\in\Phi^+\cup(-\Phi^+),1\leq r\leq d_\beta}.
\]

A parameter system is a collection
\[
\mathfrak Q
=
(\mathfrak q_{\beta,r})_{(\beta,r)\in\mathcal R_{\mathrm{hr}}}
\]
of nonzero fractional $\mathcal O_F$-ideals. Define
\[
\Gamma(\mathfrak Q)
=
\left\langle
u_{\beta,r}(x):
x\in\mathfrak q_{\beta,r},\
(\beta,r)\in\mathcal R_{\mathrm{hr}}
\right\rangle
\leq G(F).
\]

The method of this paper, together with Theorem~\ref{thm: horospherical lattices}, gives the following consequence.

\begin{theorem}
\label{thm: general-higher-rank}
Assume that $\operatorname{rank}_F(G)\geq 2$. For every parameter
system $\mathfrak Q$ as above, let
\[
S(\mathfrak Q)=\left\{
v:\Gamma(\mathfrak Q) \text{ is unbounded in }G(F_v)\right\}.
\]
Then $S(\mathfrak Q)$ is a finite set of places containing all Archimedean places, and $\Gamma(\mathfrak Q)$ is $S(\mathfrak Q)$-arithmetic. In particular, the set of places defining the $S$-arithmetic structure is precisely the set of places at which $\Gamma(\mathfrak Q)$ is unbounded in the corresponding completion.
\end{theorem}

\begin{proof}
Let $S_\infty$ denote the set of Archimedean places of $F$. We first show that $\Gamma(\mathfrak Q)$ contains an arithmetic subgroup of
$G(F)$.

For each $(\beta,r)\in\mathcal R_{\mathrm{hr}}$, set
\[
\mathfrak s_{\beta,r}=\mathfrak q_{\beta,r}\cap\mathcal O_F.
\]
Since $\mathfrak q_{\beta,r}$ is a nonzero fractional ideal, $\mathfrak s_{\beta,r}$ is a nonzero ideal of $\mathcal O_F$. Define
\[
\Lambda^\pm= \left\langle u_{\beta,r}(x):\beta\in\pm\Phi^+,x\in\mathfrak s_{\beta,r},1\le r\le d_\beta\right\rangle.
\]
Then $\Lambda^\pm$ has finite index in
$U^\pm(\mathcal O_F)$. Indeed, this follows from the usual filtration of $U^\pm$ by root height: on each successive quotient the image of $\Lambda^\pm$ contains a nonzero $\mathcal O_F$-submodule of the
corresponding root lattice, and hence has finite index.

Set
\[
\Lambda=\langle\Lambda^+,\Lambda^-\rangle.
\]
Since $\operatorname{rank}_F(G)\geq2$, we have $\operatorname{rank}_{S_\infty}(G)\geq2$. Therefore Theorem~\ref{thm: horospherical lattices} implies that $\Lambda$ is an $S_\infty$-arithmetic subgroup of $G(F)$. By construction,
\[
\Lambda\leq\Gamma(\mathfrak Q).
\]

We next show that $\Gamma(\mathfrak Q)$ is bounded at almost every non-Archimedean place. Since $\mathcal R_{\mathrm{hr}}$ is finite and each $\mathfrak q_{\beta,r}$ is a fractional ideal, there is a finite set $T$ of non-Archimedean places such that, for every $v\notin T$,
\[
\mathfrak q_{\beta,r}\subseteq\mathcal O_{F_v}\qquad\text{for all }(\beta,r)\in\mathcal R_{\mathrm{hr}}.
\]
Hence
\[
\Gamma(\mathfrak Q)\leq G(\mathcal O_{F_v})
\]
for all such $v$, and therefore $\Gamma(\mathfrak Q)$ is bounded in $G(F_v)$. Thus $S(\mathfrak Q)$ is finite.

Moreover, every Archimedean place belongs to $S(\mathfrak Q)$. For every Archimedean place $v$, the group $\Gamma(\mathfrak Q)$ contains a nontrivial $F$-root subgroup with parameters in a nonzero ideal of $\mathcal O_F$, and its image in $G(F_v)$ is unbounded. Consequently,
\[
S_\infty\subseteq S(\mathfrak Q).
\] 

We may now apply Lemma~\ref{lem: small lattice to big one} with
\[
k=F,\qquad
\Theta=\Gamma(\mathfrak Q),\qquad
\Lambda=\langle\Lambda^+,\Lambda^-\rangle.
\]
The subgroup $\Lambda$ is $S_\infty$-arithmetic, $\Lambda\leq\Gamma(\mathfrak Q)$, and $\Gamma(\mathfrak Q)$ is bounded at almost all places. By definition, the set of places at which its projection is unbounded is exactly $S(\mathfrak Q)$. Lemma~\ref{lem: small lattice to big one} therefore implies that $\Gamma(\mathfrak Q)$ is
$S(\mathfrak Q)$-arithmetic.
\end{proof}

Thus, unlike the effective description for $\SL(n)$, the $S$-arithmeticity statement itself does not require an explicit diagonal conjugation. The relevant set of places has the intrinsic
description
\[
S(\mathfrak Q)=\left\{v:\Gamma(\mathfrak Q)\text{ is unbounded in }G(F_v)\right\}.
\]
Determining this set directly from the parameter system is a separate local problem. 

When $G$ is split, one may attempt to detect these places through the valuation inequalities arising from commutator relations among the root groups, as in the $\SL(n)$ case.

Until the end of this subsection, we assume that $G$ is split over $F$. The set of places $S(\mathfrak{Q})$ is expected to be determined by the places at which the local parameter system is unbounded. As in Section~\ref{sec: find N and g}, one can formulate this in terms of negative cycles, with the difference that only paths arising from nontrivial iterated commutator relations among relative root groups should be allowed.

There is, however, a new difficulty. In the case of $\SL(n)$, the diagonal conjugation problem reduces to the system of difference constraints in equation~\eqref{eq: kill p}. For a general relative root system, the corresponding inequalities
need not form a system of difference constraints. Instead, one is led
to a system of integral linear inequalities
\[
A\vec{x}\leq\vec{b}.
\]
If this system is infeasible, Farkas' lemma provides nonnegative
rational numbers $\lambda_1,\ldots,\lambda_r$, not all zero, such that
\[
\sum_{i=1}^r\lambda_i A_i=0
\qquad\text{and}\qquad
\sum_{i=1}^r\lambda_i b_i<0.
\]
After clearing denominators, the $\lambda_i$ may be taken to be
nonnegative integers. In the present setting, the relation
\[
\sum_{i=1}^r\lambda_i A_i=0
\]
has a graphical interpretation as a balanced integral flow in the
graph of allowable root relations. This flow decomposes into directed
cycles, and the negativity of the total weight implies that at least
one of these cycles has negative weight. Thus infeasibility is again
detected by a negative cycle, although the relevant cycles are more
general than those arising from ordinary difference constraints.

However, even when the system has a solution, a resulting solution may be rational rather than integral. Thus, even when $\mathcal O_F$ is a PID, one may need to pass to a finite extension of $F$ in order to realize the required conjugation by an algebraic element.

\begin{example}
Consider the root system $B_2$, realized, for instance, by the group $\operatorname{SO}(2,3)$. Let $\alpha$ and $\beta$ be the two simple roots. The following function on the roots is admissible and has no negative cycles:
\[
\begin{array}{c|cccccccc}
&\alpha&-\alpha&\beta&-\beta&\alpha+\beta&-\alpha-\beta&2\alpha+\beta&-2\alpha-\beta\\
\hline
f&0&1&0&0&0&1&-1&1
\end{array}
\]
A direct computation shows that the corresponding system of inequalities has a unique solution
\[
h(\alpha)=\tfrac{1}{2},\qquad h(\beta)=0.
\]
Thus, a diagonal conjugation need not exist over the original field.
\end{example}

The effective congruence classification becomes substantially more delicate in this generality. When the congruence kernel is nontrivial, finite-index subgroups need not be congruence subgroups. At present, the author does not know a general procedure for deciding, from a given set of generators, whether the resulting finite-index subgroup is congruence. For this reason, the explicit classification method developed for $\SL(n)$ does not directly extend to all higher-rank groups.

Finally, generalizations of Theorem~\ref{thm: congruence conditions} to other types of Chevalley groups can be found in the survey \cite{MR687837}.
 
\section{Proof of Theorem~\ref{Thm: characteristic S-arithmetic}}\label{Sec: proof of theorem B}
In this section, we prove Theorem~\ref{Thm: characteristic S-arithmetic}. Since the rank-one case is substantially less rigid, we work only over $\mathbb Q$ rather than over general fields. The proof follows the same general path as in the higher-rank case, but we give the details here.
\subsection{\texorpdfstring
  {Step 1: $\Gamma_q$ is an $S$-arithmetic lattice}
  {Step 1: Gamma q is an S-arithmetic lattice}}
\label{Sec: Step 1} 
In this section, we prepare for the proof of Theorem~\ref{Thm: characteristic S-arithmetic} by showing that $\Gamma_q$ contains sufficiently many unipotent elements to apply Theorem~\ref{thm: horospherical lattices}. The strategy follows the method used in the proof of \cite[Theorem~3.10]{MR5043742}. The only difference is that we use Lemma~\ref{Lemma: 1/t} in place of the Greenberg--Shalom Hypothesis in the proof of \cite[Proposition~3.1]{MR5043742}.

Before beginning the proof, note that $\Gamma_q=\Gamma_{-q}$. Therefore, we may assume throughout that $q\geq 0$. 

\subsubsection{Producing arithmetic subgroups from unipotents}
In this subsection, we give a sufficient condition ensuring that $\Gamma_q\cap \SL(2,\mathbb Z)$ is a finite-index subgroup of $\SL(2,\mathbb Z)$. 

The following lemma shows that the existence of a single non-integral unipotent element allows us to construct a horospherical lattice and hence apply Theorem~\ref{thm: horospherical lattices}.

Recall that $q=\tfrac{s}{t}\in\mathbb Q$.
\begin{lemma}\label{Lemma: 1/t}
  Let $t'\geq 2$ be an integer. If 
  \[
  M=\begin{pmatrix}
      1 & \tfrac{1}{t'}\\
      0&1
  \end{pmatrix}\in \Gamma_q,
  \]
  then $\Gamma_q$ contains a finite-index subgroup of $\SL\left(2,\mathbb Z\left[\tfrac{1}{t'}\right]\right)$. In particular, $\Gamma_q\cap\SL(2,\mathbb Z)$ has finite index in $\SL(2,\mathbb Z)$.
\end{lemma}
\begin{proof}
Since $\Gamma_q$ is a subgroup of $\Gamma_1^{(t)}(s)$, we have $\Pi(t')\subset\Pi(t)$. By taking a sufficiently large power of $M$, we may assume that $t'\mid t$ without changing $\Pi(t')$.

Let \[
W_1=\begin{pmatrix}
        \tfrac{1}{\sqrt{t'}}&0\\
        0&\sqrt{t'}
    \end{pmatrix}. 
    \]
Since $W_1^{-1}MW_1=A$ and
$W_1^{-1}B_q^{t'}W_1=B_q$, the matrix $W_1$ conjugates the pair $(M,B_q^{t'})$ to $(A,B_q)$. Hence
\[
\Gamma_q\leq W_1^{-1}\Gamma_qW_1.
\]
Because $M\in\Gamma_q$, it follows that $M\in W_1^{-1}\Gamma_qW_1$. Therefore, 
    \[
    W_1MW_1^{-1}=\begin{pmatrix}
        1 &\tfrac{1}{t'^2}\\
        0 &1
    \end{pmatrix}\in \Gamma_q.
    \]
    
    Assume that for an integer $n<0$,
    \[
    \begin{pmatrix}
        1 &\tfrac{1}{t'^n}\\
        0 &1
    \end{pmatrix}\in \Gamma_q.
    \]
   The same argument, with $W_1$ replaced by
    \[
    W_n=\begin{pmatrix}
        (t')^{-\tfrac{n}{2}}& 0\\
        0 & (t')^{\tfrac{n}{2}}
    \end{pmatrix}.
    \]
    shows that 
    \[
    \begin{pmatrix}
        1 &\tfrac{1}{(t')^{2n}}\\
        0 &1
    \end{pmatrix}\in \Gamma_q.
    \]
    
    Hence, by induction, for all $n\in\mathbb Z$,
    \[
    \begin{pmatrix}
        1 &\tfrac{1}{t'^n}\\
        0 &1
    \end{pmatrix}\in \Gamma_q.
    \]
    Thus, 
   \[ U^+\left(\mathbb Z\left[\tfrac1{t'}\right]\right)
\leq\Gamma_q.\]

Since
\[
B_q^t=
\begin{pmatrix}
1&0\\
s&1
\end{pmatrix}
\in\Gamma_q,
\]
the result of Benoist--Oh \cite{MR2718935} already implies that
$\Gamma_q$ contains a finite-index subgroup of
$\SL\left(2,\mathbb Z\left[\tfrac1{t'}\right]\right)$.
However, in order to unify the rank-one and higher-rank arguments
without introducing an additional arithmeticity theorem, we construct
a second unipotent lattice directly and then apply
Theorem~\ref{thm: horospherical lattices}.

By the choice of $t'$, we have
$\Pi(t')\subseteq\Pi(t),$
and hence $\gcd(s,t')=1$. Therefore, there exist integers $N\geq1$
and $m$ such that
$t'^N=1+sm.$
Then
\[
V=
\begin{pmatrix}
1&m\\
0&1
\end{pmatrix}
\begin{pmatrix}
1&0\\
s&1
\end{pmatrix}
\begin{pmatrix}
1&-\tfrac{m}{t'^N}\\
0&1
\end{pmatrix}
=
\begin{pmatrix}
t'^N&0\\
s&t'^{-N}
\end{pmatrix}
\in\Gamma_q.
\]
Conjugating $B_q^t$ by powers of $V$, we obtain
\[
\begin{pmatrix}
1&0\\
\tfrac{s}{t'^{2Nk}}&1
\end{pmatrix}
\in\Gamma_q
\qquad
\text{for every }k\geq0.
\]
Thus,
\[
U^-\left(s\mathbb Z\left[\tfrac1{t'}\right]\right)
\leq\Gamma_q.
\]
Together with the previously constructed upper-unipotent lattice,
Theorem~\ref{thm: horospherical lattices} implies that $\Gamma_q$
contains a finite-index subgroup of
$\SL\left(2,\mathbb Z\left[\tfrac1{t'}\right]\right)$.

\end{proof}
\subsubsection{\texorpdfstring{$\Gamma_q$}{\Gamma_q} is an \texorpdfstring{$S$}{S}-arithmetic lattice}
In this section, we prove the first step towards the proof of Theorem~\ref{Thm: characteristic S-arithmetic}.

The key idea is that an element of $U_q\setminus H$ forces the existence of a non-integral unipotent element in $\Gamma_q$. By Lemma~\ref{Lemma: 1/t}, the presence of such an element implies that $\Gamma_q$ contains a horospherical lattice and is therefore $S$-arithmetic. Thus, the proof of Theorem~\ref{Them: finite index} reduces to extracting an appropriate unipotent element from the normalizer condition.
\begin{theorem}\label{Them: finite index}
    Let $q>0$ with $q\neq 3$. If $U_q\neq H$, then $\Gamma_q$ has finite index in $\SL\left(2,\mathbb Z\left[\tfrac{1}{t}\right]\right)$.
\end{theorem}

The proof proceeds by showing that a nontrivial element of $U_q$ produces a non-integral unipotent element of $\Gamma_q$, after which Lemma~\ref{Lemma: 1/t} yields arithmeticity.
\begin{proof}
Recall that when $q\geq 4$, Brenner \cite{MR75952} shows that $\Gamma_q$ is free, and in particular, $U_q=H$. Hence, we may assume throughout that $0<q<4$.

 The argument naturally divides into three cases. 

\textbf{Case 1: \texorpdfstring{$q$}{q} is an integer.}
The relevant results are classical. We list them here for completeness.

When $q= 1$ or $2$, $\Gamma_q=\Gamma_1(q)$ and \[
U_q=H\times\begin{pmatrix}
    \pm 1&0\\0&\pm 1
\end{pmatrix}.
\]

When $q=3$, $\Gamma_q=\Gamma_1(3)$ and $U_q=H$.

\textbf{Case 2: the setting of Corollary \ref{Cor: t=kspm1}.}
Let $q=\tfrac{s}{t}\in \mathbb Q\setminus{\mathbb Z}$. Assume that there exists $k\in \mathbb Z$ such that $t=ks\pm 1$. Changing the sign of $t$ if necessary, we may assume that $t=ks+1$. It follows that 
\begin{equation}\label{eq: t=kspm1}
    X=B^{-1}A^kB^t=\begin{pmatrix}
    t &k\\
    0& \tfrac{1}{t}
\end{pmatrix}\in \Gamma_q.
\end{equation}
Hence, for any $n\in \mathbb Z$,
\[
X^{-n}AX^{n}=\begin{pmatrix}
    1 &\tfrac{1}{t^{2n}}\\
    0 &1
\end{pmatrix}\in \Gamma_q.
\]

By Theorem~\ref{thm: horospherical lattices}, $\Gamma_q$ is a finite-index subgroup of $\SL\left(2,\mathbb Z\left[\tfrac{1}{t}\right]\right)$.

In particular, the theorem applies to all cases with $s=1,2,3,4,6$. This phenomenon is not entirely new: special cases of equation~\eqref{eq: t=kspm1} were already observed in \cite[Proposition 7]{MR258975}, and a closely related equation appears in the proof of \cite[Theorem 2]{MR1370894}.

\textbf{Case 3: the general case.} 
Let $q=\tfrac{s}{t}\in\mathbb Q\setminus\mathbb Z$, where $s\geq3$ and $q<4$, and assume that $U_q\neq H$. Choose $V\in U_q\setminus H$. Then \[
V=\begin{pmatrix}
    a& b\\
    0&\tfrac{1}{a}
\end{pmatrix}.
\]

We distinguish two situations. 

Assume that $a$ and $\tfrac{1}{a}$ are both integers. As $\Gamma_q$ is a subgroup of $\Gamma_1^{(t)}(s)$, we have 
\[
a\equiv_t 1 \pmod s.
\]
Since $s\geq 3$, it follows that $a=1$. Therefore, $b=\tfrac{s'}{t'}\in \mathbb Q\setminus{\mathbb Z}$. By Bézout's identity, there exist $u, v\in \mathbb Z$ such that
\[
\begin{pmatrix}
    1 & \tfrac{1}{t'}\\
    0&1
\end{pmatrix}=A^uV^v\in \Gamma_q.
\]

Otherwise, replacing $a$ by $a^{-1}$ if necessary, we may assume that
$a\notin\mathbb Z$. Write
\[
a=\tfrac{m}{n}
\]
in lowest terms. We have
\[
\begin{pmatrix}
    1& \tfrac{m^2}{n^2}\\
    0&1
\end{pmatrix}=VAV^{-1}\in \Gamma_q.
\]
By Bézout's identity,
\[
\begin{pmatrix}
    1&\tfrac{1}{n^2}\\
    0&1
\end{pmatrix}\in \Gamma_q.
\]

In both situations, Lemma~\ref{Lemma: 1/t} gives
\begin{equation}\label{eq: Gamma_q contains a lattice}
    \left[\SL(2,\mathbb Z):\Gamma_q\cap\SL(2,\mathbb Z)\right]<\infty.
\end{equation}

\medskip
 
The remainder of the proof follows the strategy of \cite[Theorem 3.10]{MR5043742}. We include the argument for completeness. Note that it is essentially the same as Lemma~\ref{lem: small lattice to big one}; however, the version here is more adapted to the question of $\Gamma_q$.

The group $\Gamma_q$ is a discrete subgroup of $L=\SL(2,\mathbb{R})\times\prod_{p\in \Pi(t)}\SL(2,\mathbb Q_p)$. Set $\mathbb Q_\infty=\mathbb R$. For every subset $S\subseteq\{\infty\}\cup\Pi(t)$, define
\[
L_S=\prod_{p\in S}\SL(2,\mathbb Q_p),
\]
and let $\operatorname{pr}_S:L\to L_S$ denote the projection.

 Let $T\subset\{\infty\}\cup\Pi(t)$ be a minimal subset such that $\Gamma_q$ is discrete in 
 \[L_T=\prod_{p\in T}\SL(2,\mathbb Q_p).\] The set $T$ contains at least two places. Indeed, the projection of
$\Gamma_q$ to $\SL(2,\mathbb R)$ is not discrete by the work of Knapp \cite{MR248231}, while its projection to $\SL(2,\mathbb Q_p)$ is not discrete because $A^{p^r}\to I$ in $\SL(2,\mathbb Q_p)$. Moreover, these projections are Zariski dense: their Zariski closures cannot be solvable, and neither $\SL(2,\mathbb R)$ nor $\SL(2,\mathbb Q_p)$ contains a proper Zariski-closed nonsolvable subgroup. Therefore $\operatorname{pr}_T(\Gamma_q)\subset L_T$ is irreducible. 

We continue the proof by applying \cite[Lemma~A.18]{MR5043742} with
\[
\Gamma=\Gamma_q,\qquad
G_0=\SL(2,\mathbb R),\qquad
G_{\mathrm{na}}
=\prod_{p\in T\setminus\{\infty\}}\SL(2,\mathbb Q_p).
\]
Then
\[
\operatorname{pr}_0(\Gamma_{\mathrm{na}})
=\Gamma_q\cap\SL(2,\mathbb Z).
\]
By equation~\eqref{eq: Gamma_q contains a lattice} and
\cite[Lemma~A.18]{MR5043742}, $\Gamma_q$ is a lattice in $L_T$.
Therefore $\Gamma_q$ is commensurable with $\SL\left(2,\mathbb Z\left[\tfrac{1}{t'}\right]\right)$ where $t'=\prod_{p\in T\setminus\{\infty\}}p$. 

To see that $T=\Pi(t)\cup \{\infty\}$, it suffices to show that $\Gamma_q$ has unbounded projection
to $\SL(2,\mathbb Q_p)$ for all $p\in \Pi(t)$. Indeed, $A^{-1}B_q$ has trace $2-q=2-\tfrac{s}{t}$, and is
therefore not elliptic in $\SL(2,\mathbb Q_p)$. In particular, $A^{-1}B_q$ generates an unbounded
subgroup.

In summary, we have shown that for every $q>0$ with $q\neq 3$, if $U_q\neq H$, then $\Gamma_q$ is of finite index in $\SL\left(2,\mathbb Z\left[\tfrac{1}{t}\right]\right)$.  
\end{proof}
\subsection{Step 2: the group structure}
In this section, we complete the proof of Theorem~\ref{Thm: characteristic S-arithmetic}. By Theorem~\ref{Them: finite index}, the group $\Gamma_q$ is an $S$-arithmetic subgroup of $\SL\left(2,\mathbb Z\left[\tfrac{1}{t}\right]\right)$. Using Serre's congruence subgroup theorem, we first show that $\Gamma_q$ contains a principal congruence subgroup of level supported on the primes dividing $s$. We then identify the image of $\Gamma_q$ modulo $s^\lambda$, which allows us to conclude that $\Gamma_q=\Gamma_1^{(t)}(s)$.

Let $q=\tfrac{s}{t}\in\mathbb Q$, and assume that $U_q\neq H$.

When $q\in\mathbb Z$, Theorem~\ref{Thm: characteristic S-arithmetic} is a classical result. 

Assume that $q\notin\mathbb Z$. By Theorem~\ref{Them: finite index}, $\Gamma_q$ is a finite-index subgroup of $\SL\left(2,\mathbb Z\left[\tfrac{1}{t}\right]\right)$. 
For $n\in\mathbb N^+$, define \[
\Gamma_0^{(t)}(n)=\setdef{\begin{pmatrix}
    a&b\\
    c&d
\end{pmatrix}}{a\equiv_td\equiv_t1 \pmod n, \quad b\equiv_tc\equiv_t 0 \pmod n}.
\] Note that $\Gamma_0^{(t)}(n)=\ker \left(\operatorname{mod}_n:\SL\left(2,\mathbb Z\left[\tfrac{1}{t}\right]\right)\to\SL(2, \mathbb Z/n\mathbb Z)\right)$.
Then $\Gamma_q$ contains $\Gamma_0^{(t)}(n)$ for some $n$ by Theorem~\ref{Thm: CSP}.

The first step is to show that the congruence level can be chosen to involve only primes dividing $s$. 

Let $\Omega$ be a minimal set of primes for which there exists an integer $n$ satisfying $\Pi(n)=\Omega$ and
\[
\Gamma_0^{(t)}(n)\leq \Gamma_q.
\]
Suppose, for contradiction, that $\Omega\setminus\Pi(s)\neq\emptyset$. Choose a prime $p\in \Omega\setminus\Pi(s)$, and write $n=p^m n'$, where $p$ and $n'$ are coprime. By the Chinese Remainder Theorem,
\[
\SL(2,\mathbb Z/n\mathbb Z) \simeq \SL(2,\mathbb Z/p^m\mathbb Z)\times\SL(2,\mathbb Z/n'\mathbb Z).
\]

Let
\[
\operatorname{mod}_{p^m}:\SL\left(2,\mathbb Z\left[\tfrac1t\right]\right)\to \SL(2,\mathbb Z/p^m\mathbb Z)
\]
be the reduction map modulo $p^m$.

Since $s$ and $p$ are coprime, there exists $w\in\mathbb Z$ such that
$sw\equiv 1 \pmod{p^m}$.
Hence $\operatorname{mod}_{p^m}(\Gamma_q)$ contains
\[
A \pmod{p^m}
\quad\text{and}\quad
(B_q)^{tw} \equiv B_1 \pmod{p^m}.
\]
Since $A$ and $B_1$ generate $\SL(2,\mathbb Z)$ and the reduction map $\SL(2,\mathbb Z)\to\SL(2,\mathbb Z/p^m\mathbb Z)$ is surjective, it follows that
\[
\operatorname{mod}_{p^m}(\Gamma_q)
=\SL(2,\mathbb Z/p^m\mathbb Z).\]
Consequently,
\[
\Gamma_0^{(t)}(n')\leq \Gamma_q.
\]
Since
$\Pi(n')=\Omega\setminus\{p\}$,
this contradicts the minimality of $\Omega$. Therefore,
$\Omega\subset \Pi(s)$.

Therefore, there exists $\lambda\in \mathbb N$ such that $\Gamma_0^{(t)}(s^\lambda)\leq \Gamma_q$.

Let
\[
\operatorname{mod}_{s^\lambda}:\SL\left(2,\mathbb Z\left[\tfrac1t\right]\right)\to \SL(2,\mathbb Z/s^\lambda\mathbb Z)
\]
be the reduction map modulo $s^\lambda$, and let
\[
\operatorname{Mod}_s:\SL(2,\mathbb Z/s^\lambda\mathbb Z)\to \SL(2,\mathbb Z/s\mathbb Z)
\]
be the natural projection.

Having reduced to congruence levels supported on $s$, we now reformulate the desired equality $\Gamma_q=\Gamma_1^{(t)}(s)$ in terms of the reduction map modulo $s^\lambda$.

The equality $\Gamma_q=\Gamma_1^{(t)}(s)$ is equivalent to 
\[
\operatorname{mod}_{s^\lambda}(\Gamma_q)=(\operatorname{Mod}_s)^{-1}
\left(
\left\{
\begin{pmatrix}
1 & *\\
0 & 1
\end{pmatrix}
\right\}\right).
\]
Thus, the proof reduces to verifying that the image of $\Gamma_q$ modulo $s^\lambda$ is exactly the preimage of the standard upper-unitriangular subgroup of $\mathrm{SL}(2,\mathbf Z/s\mathbf Z)$.

Therefore, to prove Theorem~\ref{Thm: characteristic S-arithmetic}, it is sufficient to verify the following two conditions:
\begin{enumerate}
    \item\label{Condition 1} $\ker(\operatorname{Mod}_s) \subset\operatorname{mod}_{s^\lambda}(\Gamma_q)$.
    \item\label{condition 2} $\operatorname{Mod}_s\!\bigl(\operatorname{mod}_{s^\lambda}(\Gamma_q)\bigr) = \left\{ \begin{pmatrix} 1 & *\\ 0 & 1 \end{pmatrix} \right\}$.
\end{enumerate}

Condition~\eqref{condition 2} follows from the fact that $\operatorname{Mod}_s(\operatorname{mod}_{s^\lambda}(B_q))=I_2\pmod{s}$. Let 
\[C=\begin{pmatrix}
    a &b\\
    c &d
\end{pmatrix}\in \ker(\operatorname{Mod}_s).\]
Then $a-1\equiv d-1\equiv c\equiv b\equiv 0\pmod{s}$. Choose $\lambda_1$ such that $a\lambda_1+b\equiv 1\pmod{s^\lambda}$. Then 
\[
CA^{\lambda_1}=\begin{pmatrix}
    a &1\\
    c &d_1
\end{pmatrix}\pmod{s^\lambda}.
\]
Let $\mu_1=\tfrac{1-a}{s}$ and $\mu_2=\tfrac{1-d_1}{s}$. We have
\[
B_q^{t\mu_2}CA^{\lambda_1}B_q^{t\mu_1}=\begin{pmatrix}
    1& 1\\
    c_1&1
\end{pmatrix}\pmod{s^{\lambda}}.
\]
Since $\det\begin{pmatrix}
    1& 1\\
    c_1&1
\end{pmatrix}\equiv 1\pmod{s^\lambda}$, it follows that $c_1\equiv 0\pmod{s^\lambda}$. Hence,
\[
A^{-1}B_q^{t\mu_2}CA^{\lambda_1}B_q^{t\mu_1}\in \Gamma_{0}^{(t)}(s^\lambda).
\]
Therefore $C\in \Gamma_q$. This verifies Condition~\eqref{Condition 1} and completes the proof of Theorem~\ref{Thm: characteristic S-arithmetic}.

We conclude the section by proving Corollary~\ref{Cor: t=kspm1}.
\begin{proof}[Proof of Corollary~\ref{Cor: t=kspm1}]
By equation~\eqref{eq: t=kspm1}, $U_q\neq H$. Therefore, $q\in\mathcal{CS}$ by Theorem~\ref{Thm: characteristic S-arithmetic}.
\end{proof}
\section{General rank-one groups}\label{sec:rank-one-general}
We discuss an application of Theorem~\ref{Thm: characteristic S-arithmetic} to simple real Lie groups of real rank one.

Such groups are classified, up to isogeny, into three infinite
families and one exceptional group:
\begin{enumerate}
    \item $\operatorname{PSO}(n,1)$;
    \item $\operatorname{PU}(n,1)$;
    \item $\operatorname{PSp}(n,1)$;
    \item $F_4^{-20}$.
\end{enumerate}

Let $\mathbf G$ be a fixed $\mathbb Q$-form of one of these groups.
Choose a composition algebra $K_0$ over the rational numbers compatible with the
standard rank-one model, so that
\[
K_0\otimes_{\mathbb Q}\mathbb R
\cong
K,
\qquad
K\in\{\mathbb R,\mathbb C,\mathbb H,\mathbb O\},
\]
and fix a maximal $\mathbb Z$-order
\[
\mathcal A\subset K_0.
\]

Set
\[
r=
\begin{cases}
n-1,
&
\mathbf G(\mathbb R)
\text{ is locally isomorphic to }
\operatorname{PSO}(n,1),
\operatorname{PU}(n,1),
\text{ or }\operatorname{PSp}(n,1),\\[2mm]
1,
&
\mathbf G(\mathbb R)
\text{ is locally isomorphic to }F_4^{-20}.
\end{cases}
\]

The positive unipotent Lie algebra admits a decomposition
\[
\mathfrak u^+
\cong
K_0^r\oplus\operatorname{Im}(K_0)
\]
as a $\mathbb Q$-vector space. Equivalently, $U^+$ may be parametrized
by pairs
\[
(v,s)\in K_0^r\oplus\operatorname{Im}(K_0),
\]
with a two-step nilpotent group law. The first summand corresponds to
the restricted root $\alpha$, and the second corresponds to the
restricted root $2\alpha$. When $K=\mathbb R$, the second summand is
trivial.

Let $\mathcal L(K_0)$ denote the set of $\mathbb Z$-lattices in
\[
K_0=\mathcal A\otimes_{\mathbb Z}\mathbb Q,
\]
and let $\mathcal L(\operatorname{Im}(K_0))$ denote the set of
$\mathbb Z$-lattices in $\operatorname{Im}(K_0)$.

Define
\[
\mathcal I_{\mathrm{rk1}}
=
\{\pm(\alpha,i):1\leq i\leq r\}
\cup
\mathcal I_{2\alpha},
\]
where
\[
\mathcal I_{2\alpha}
=
\begin{cases}
\{\pm(2\alpha)\},&\operatorname{Im}(K_0)\neq0,\\
\varnothing,&\operatorname{Im}(K_0)=0.
\end{cases}
\]

For $1\leq i\leq r$, let
\[
u_{\pm(\alpha,i)}:K_0\longrightarrow U_{\pm\alpha}
\]
be the corresponding horospherical coordinate maps. When
$\operatorname{Im}(K_0)\neq0$, let
\[
u_{\pm2\alpha}:\operatorname{Im}(K_0)
\longrightarrow U_{\pm2\alpha}
\]
be the coordinate maps for the central root groups.

A parameter system is a collection
\[
\mathfrak Q=(L_\delta)_{\delta\in\mathcal I_{\mathrm{rk1}}},
\]
where
\[
L_\delta\in\mathcal L(K_0)
\quad\text{if}\quad
\delta=\pm(\alpha,i),
\]
and
\[
L_\delta\in\mathcal L(\operatorname{Im}(K_0))
\quad\text{if}\quad
\delta=\pm2\alpha.
\]

For each $\beta\in\mathcal I_{\mathrm{rk1}}$, let
\[
V_\beta
=
\begin{cases}
K_0, & \text{if }\beta=\pm(\alpha,i),\\[1mm]
\operatorname{Im}(K_0), & \text{if }\beta=\pm2\alpha.
\end{cases}
\]
Choose the opposite-root coordinate maps
\[
u_\beta:V_\beta\longrightarrow U_\beta
\qquad\text{and}\qquad
u_{-\beta}:V_{-\beta}\longrightarrow U_{-\beta}.
\]

For $x\in V_\beta$ and $y\in V_{-\beta}$, we say that $(x,y)$ is an
$\SL_2$-pair if there exist a homomorphism
\[
\iota_{x,y}:\SL(2,\mathbb R)\longrightarrow G
\]
and a nonzero real number $q$ such that
\[
\iota_{x,y}
\begin{pmatrix}
1&1\\
0&1
\end{pmatrix}
=
u_\beta(x)
\]
and
\[
\iota_{x,y}
\begin{pmatrix}
1&0\\
q&1
\end{pmatrix}
=
u_{-\beta}(y).
\]
With respect to the chosen normalization of the opposite-root coordinates, the number $q$ is uniquely determined by $(x,y)$. Define
\[
c_\beta:V_\beta\times V_{-\beta}\longrightarrow\mathbb R
\]
by
\[
c_\beta(x,y)
=
\begin{cases}
q, & \text{if $(x,y)$ is an $\SL_2$-pair with parameter $q$},\\
0, & \text{otherwise}.
\end{cases}
\]
For a parameter system
\[
\underline{\mathfrak Q}
=
(L_\beta)_{\beta\in\mathcal I_{\mathrm{rk1}}},
\]
define
\[
Q_\beta
=
\left\{
c_\beta(x,y):
x\in L_\beta,\ y\in L_{-\beta}
\right\}
\cap\mathbb Q^\times.
\]

Define
\[
\Gamma(\mathfrak Q)
=
\langle
u_\beta(t):t\in L_\beta,\ \beta\in\mathcal I_{\mathrm{rk1}}
\rangle.
\]

The following is a particularly simple case to which our method applies.

\begin{theorem}
Let $\mathfrak Q$ be a parameter system of a simple real Lie group $G$ of rank one. Suppose that
\[
Q_\beta\cap\mathcal{CS}\neq\emptyset
\]
for some $\beta\in\mathcal I_{\mathrm{rk1}}$. Then $\Gamma(\mathfrak Q)$ is $S$-arithmetic for some finite set of places $S$.
\end{theorem}

\begin{proof}
Let $q=\tfrac{s}{t}\in Q_\delta\cap\mathcal{CS}$, and choose
$x\in L_\delta$ and $y\in L_{-\delta}$ such that
$c_\delta(x,y)=q$. By the normalization of the opposite-root
coordinates, the subgroup generated by $u_\delta(x)$ and
$u_{-\delta}(y)$ is identified with $\Gamma_q$ inside an embedded
copy of $\SL(2,\mathbb R)$.

Since $q\in\mathcal{CS}$, this subgroup contains the image of
\[
\begin{pmatrix}
t^k&0\\
0&t^{-k}
\end{pmatrix}
\]
for some $k>0$. Denote this image by $a$. Then $a$ belongs to the
one-dimensional split torus $T$, and conjugation by $a$ acts on the
restricted root groups $U_{\pm\alpha}$ and $U_{\pm2\alpha}$ by
nontrivial integral powers of $t$.

Because every $L_\eta$ is a $\mathbb Z$-lattice and
$\mathcal I_{\mathrm{rk1}}$ is finite, $\Gamma(\mathfrak Q)$ contains finite-index subgroups of
both
\[
U^+\left(\mathbb Z\left[\tfrac{1}{t}\right]\right)
\qquad\text{and}\qquad
U^-\left(\mathbb Z\left[\tfrac{1}{t}\right]\right).
\]
 Theorem~\ref{thm: horospherical lattices} therefore implies that
$\Gamma(\mathfrak Q)$ is $S$-arithmetic for some finite set of
places $S$.
\end{proof}
\section{Small virtually generating sets for finite-index subgroups}
\label{Sec: 2-generated}

In this section, we prove
Theorems~\ref{THM: SL(n Z) rank gradient}
and~\ref{Thm: SL(2,O_F) rank gradient}, together with a
generalization to irreducible nonuniform $S$-arithmetic lattices of
higher $S$-rank. In the general setting, the constant $2$ is replaced
by an integer $k=k(G)$ depending only on the ambient algebraic group
$G$. Our argument does not give a sharp value of $k(G)$, so we state
the result in a form convenient for the proof, without attempting to
optimize this constant.

Before turning to the main result, we recall a number-theoretic
consequence of the work of Sarma~\cite{MR2282362}. To state it, suppose that $F$ is a CM field, and let $F^+$ denote its maximal totally real subfield. Write
\[
F=F^+(\sqrt{-\delta}),
\]
where $\delta\in\mathcal O_{F^+}$ is totally positive. Let $S$ be a finite set of places of $F$ containing all Archimedean places, and denote by $S_\infty$ the set of Archimedean places of $F$. We identify each non-Archimedean place of $F$ with the corresponding nonzero prime ideal of $\mathcal O_F$. Define
\[
S(F^+)=\{\mathfrak p\cap\mathcal O_{F^+}:\mathfrak p\in S\setminus S_\infty\}\cup S_\infty(F^+),
\]
where $S_\infty(F^+)$ denotes the set of Archimedean places of $F^+$.
\begin{lemma}\cite[Theorem~2, Corollary~1, Lemma~10]{MR2282362}
\label{lem: Sarma units}
Let $F$ be a number field and let $S$ be a finite set of places containing all Archimedean places, with $|S|\ge2$. Then one of the following holds.
\begin{enumerate}
\item There exists $\alpha\in\mathcal O_{F,S}^{\times}$ such that $\mathbb Z[\alpha^n]$ has finite index in the additive group of $\mathcal O_{F,S}$ for every $n\ge1$.

\item The field $F$ is CM, say $F=F^+(\sqrt{-\delta})$, and there exists
\[
\alpha\in\mathcal O_{F^+,S(F^+)}^\times\cap\mathcal O_{F,S}^\times
\]
such that $\mathbb Z[\alpha^n][\sqrt{-\delta}]$ has finite index in the additive group of $\mathcal O_{F,S}$ for every $n\ge1$.
\end{enumerate}
\end{lemma}

\subsection{The general case}

We first state the general result.

\begin{theorem}\label{thm: gradiant rank general}
Let $F$ be a number field, and let $G$ be a connected, absolutely almost simple algebraic group defined over $F$ with $\operatorname{rank}_F(G)\ge1$. Let $S$ be a finite set of places of $F$ containing all Archimedean places. If $\operatorname{rank}_S(G)\ge2$, then $G(\mathcal O_{F,S})$ has arbitrarily small $k(G)$-generated finite-index subgroups for some constant $k(G)$ depending only on $G$.
\end{theorem}

\begin{proof}
Let $\Gamma<G(\mathcal O_{F,S})$ be a finite-index subgroup.

We first dispose of the case $S=S_\infty$. Since $\operatorname{rank}_F(G)\ge1$, the arithmetic lattice $G(\mathcal O_F)$ is nonuniform, and since $\operatorname{rank}_{S_\infty}(G)\ge2$, it is a higher-rank arithmetic lattice. By \cite{MR2174097}, $\Gamma$ has arbitrarily small three-generated finite-index subgroups. Thus, we may assume for the remainder of the proof that $S\setminus S_\infty\neq\emptyset$. In particular, $|S|\ge2$.

\textbf{Case 1: $\operatorname{rank}_F(G)=1$.}
Choose an $\SL_2$-pair corresponding to a pair of opposite relative root directions. More precisely, choose $(\beta,r),(-\beta,r')\in\mathcal R_{\mathrm{hr}}$ and an $F$-homomorphism
\[
\iota:\SL_2\longrightarrow G
\]
such that
\[
\iota
\begin{pmatrix}
1&x\\
0&1
\end{pmatrix}
\in u_{\beta,r}(F),
\qquad
\iota
\begin{pmatrix}
1&0\\
x&1
\end{pmatrix}
\in u_{-\beta,r'}(F)
\]
for $x\in F$, and such that
\[
\iota
\begin{pmatrix}
a&0\\
0&a^{-1}
\end{pmatrix}
\in T(F),
\qquad a\in F^\times,
\]
where $T$ is the fixed maximal $F$-split torus of $G$.

Let \[
\Delta=\iota\bigl(\SL_2(\mathcal O_{F,S})\bigr)\cap G(\mathcal O_{F,S}).
\]
Since $\iota$ is defined over $F$, $\Delta$ is commensurable with
$\iota(\SL_2(\mathcal O_{F,S}))$. Hence
$\Gamma\cap\Delta$ is an $S$-arithmetic subgroup of
$\iota(\SL_2(F))$.

By \cite{MR2282362}, there exist
\[
\gamma_1,\gamma_2,\gamma_3\in\Gamma\cap\Delta
\]
such that
\[
\langle\gamma_1,\gamma_2,\gamma_3\rangle
\]
has finite index in $\Delta$. Let $\alpha$ be as in Lemma~\ref{lem: Sarma units}. Then there exists $n\geq 1$ such that 
\[\iota
\begin{pmatrix}
\alpha^n&0\\
0&\alpha^{-n}
\end{pmatrix}
\in
\langle\gamma_1,\gamma_2,\gamma_3\rangle.
\]

Since $\Gamma$ has finite index in $G(\mathcal O_{F,S})$, there exists
$m\ge1$ such that
\[
u_{\beta,r}(m\mathcal O_{F,S})\subseteq\Gamma
\]
for every $(\beta,r)\in\mathcal R_{\mathrm{hr}}$.

If the first alternative of Lemma~\ref{lem: Sarma units} holds, set
\[
\eta_{\beta,r}=u_{\beta,r}(m).
\]
If the second alternative holds, write
\[
F=F^+(\sqrt{-\delta})
\]
as above, and for each $(\beta,r)\in\mathcal R_{\mathrm{hr}}$ set
\[
\eta_{\beta,r}=u_{\beta,r}(m),
\qquad
\eta'_{\beta,r}=u_{\beta,r}(m\sqrt{-\delta}).
\]

Since $\operatorname{rank}_F(G)=1$, every relative root is of the form $\pm\beta_0$ or $\pm2\beta_0$, where $\beta_0$ is the indivisible positive relative root. It follows that conjugation by
\[
h=
\iota
\begin{pmatrix}
\alpha^n&0\\
0&\alpha^{-n}
\end{pmatrix}
\]
multiplies the parameter of each root subgroup by one of
\[
\alpha^{2n},\qquad \alpha^{-2n},\qquad
\alpha^{4n},\qquad \alpha^{-4n}.
\]
Conjugating the elements $\eta_{\beta,r}$ and, in the CM case,
$\eta_{\beta,r}'$ by suitable positive or negative powers of
$h$, Lemma~\ref{lem: Sarma units} implies that the resulting
subgroup intersects each of $U^+(\mathcal O_{F,S})$ and $U^-(\mathcal O_{F,S})$ in a subgroup of finite index.

Let $H$ denote the subgroup of $\Gamma$ generated by $\gamma_1,\gamma_2,\gamma_3$, together with all the $\eta_{\beta,r}$ and, when necessary, the $\eta_{\beta,r}'$. By Theorem~\ref{thm: horospherical lattices}, $H$ contains an $S$-arithmetic subgroup of $G(F)$. Since
\[
H\le\Gamma\le G(\mathcal O_{F,S}),
\]
it follows that $H$ itself has finite index in $G(\mathcal O_{F,S})$. Thus in this case we may take
\[
k(G)=3+2|\mathcal R_{\mathrm{hr}}|.
\]

\textbf{Case 2: $\operatorname{rank}_F(G)\ge2$.}

As above, choose an $F$-homomorphism
\[
\iota:\SL_2\longrightarrow G
\]
associated with one of the relative root directions.

Since
\[
G(\mathcal O_F)\leq G(\mathcal O_{F,S})
\]
and $\Gamma$ has finite index in $G(\mathcal O_{F,S})$, the intersection $\Gamma\cap G(\mathcal O_F)$ has finite index in $G(\mathcal O_F)$. By \cite{MR2174097}, there exist
\[
\gamma_1,\gamma_2,\gamma_3\in \Gamma\cap G(\mathcal O_F)
\]
such that the subgroup
\[
\Lambda:=\langle\gamma_1,\gamma_2,\gamma_3\rangle
\]
has finite index in $G(\mathcal O_F)$.

Let
\[
\Delta_S=\iota\bigl(\SL_2(\mathcal O_{F,S})\bigr)\cap G(\mathcal O_{F,S}).
\]
Then $\Gamma\cap\Delta_S$ has finite index in $\Delta_S$. Since
$|S|\ge2$, \cite{MR2282362} gives
\[
\eta_1,\eta_2,\eta_3\in\Gamma\cap\Delta_S
\]
such that
\[
\langle\eta_1,\eta_2,\eta_3\rangle
\]
has finite index in $\Delta_S$.

Let \[
\Theta=\langle\gamma_1,\gamma_2,\gamma_3,\eta_1,\eta_2,\eta_3\rangle.
\]
The subgroup $\Lambda$ is arithmetic, hence $S_\infty$-arithmetic, while $\Theta$ is bounded in $G(F_v)$ for every finite place $v\notin S$. Moreover, since
\[
\langle\eta_1,\eta_2,\eta_3\rangle
\]
has finite index in the corresponding $S$-arithmetic $\SL_2$-subgroup, $\Theta$ is unbounded in $G(F_v)$ for every $v\in S\setminus S_\infty$. Therefore, Lemma~\ref{lem: small lattice to big one} implies that $\Theta$ is $S$-arithmetic. Since
\[
\Theta\leq G(\mathcal O_{F,S}),
\]
it follows that $\Theta$ has finite index in $G(\mathcal O_{F,S})$. Thus $\Gamma$ contains a finite-index subgroup generated by at most six elements, and hence in this case we may take
\[
k(G)=6.
\]

This completes the proof.
\end{proof}
\subsection{Sharp results} 

In contrast with the general case, the bounds in
Theorems~\ref{THM: SL(n Z) rank gradient}
and~\ref{Thm: SL(2,O_F) rank gradient} are sharp.

\begin{proof}[Proof of Theorem~\ref{THM: SL(n Z) rank gradient}]
Let
\[
\Gamma<\SL\left(n,\mathbb Z\left[\tfrac1N\right]\right)
\]
be a finite-index subgroup. Then
\[
\Lambda:=\Gamma\cap\SL(n,\mathbb Z)
\]
has finite index in $\SL(n,\mathbb Z)$. By
\cite{MR3565440}, after replacing the unipotent element by a suitable positive power, there exist $h,u\in\Lambda$, with $u$ a nontrivial unipotent, such that
\[
\langle h,u^j\rangle
\]
has finite index in $\Lambda$ for every $j\in\mathbb N$.

Define the one-parameter subgroup
\[
u(t):=\exp(t\log u),\qquad t\in\mathbb Q.
\]
Since $u$ is unipotent, both $\log u$ and $\exp(t\log u)$ are finite polynomials, and in particular $u(t)\in\SL(n,\mathbb Q)$. Since $\Gamma$ has finite index in $\SL(n,\mathbb Z\left[\tfrac1N\right])$, after choosing $m\in\mathbb N$ sufficiently divisible we have
\[
u\left(m\mathbb Z\left[\tfrac1N\right]\right)\subseteq\Gamma.
\]

For each $k\in\mathbb N$, set
\[
\Theta_k
:=
\left\langle
h,u\left(\frac{m}{N^k}\right)
\right\rangle
\leq\Gamma.
\]
Since
\[
u\left(\frac{m}{N^k}\right)^{N^k}=u(m)=u^m,
\]
the subgroup
\[
\Lambda_0:=\langle h,u(m)\rangle
\]
is contained in $\Theta_k$. By the choice of $h$ and $u$, $\Lambda_0$ has finite index in $\Lambda$, and hence in $\SL(n,\mathbb Z)$. Thus, there exists $m'\in\mathbb N$ such that
\[
(u^{\mathsf T})^{m'}\in\Lambda_0\subseteq \Theta_k,
\]
where $u^{\mathsf T}$ is the transpose of $u$.

Let 
\[
P(t)=\operatorname{tr}\left((u^{\mathsf T})^{m'}u(t)\right).
\]
Then $P$ is a nonconstant polynomial with rational coefficients. It
follows that
\[
\left|
P\left(\frac{m}{N^k}\right)
\right|_p\longrightarrow\infty.
\]
Thus, for all sufficiently large $k$, the projection of $\Theta_k$
is unbounded in $\SL(n,\mathbb Q_p)$ for every $p\in\Pi(N)$. On the
other hand, since
\[
\Theta_k\leq\Gamma
\leq\SL\left(n,\mathbb Z\left[\tfrac1N\right]\right),
\]
the projection of $\Theta_k$ to $\SL(n,\mathbb Q_p)$ is bounded for
every $p\notin\Pi(N)$.

We now apply Lemma~\ref{lem: small lattice to big one} with
\[
\Lambda_0=\langle h,u(m)\rangle
\qquad\text{and}\qquad
\Theta=\Theta_k.
\]
The subgroup $\Lambda_0$ is arithmetic, $\Theta_k$ is bounded at
every finite place outside $\Pi(N)$, and, for $k$ sufficiently large,
it is unbounded at every place in $\Pi(N)$. Hence
Lemma~\ref{lem: small lattice to big one} implies that $\Theta_k$ is
$\bigl(S_\infty\cup\Pi(N)\bigr)$-arithmetic. Since
\[
\Theta_k
\leq
\SL\left(n,\mathbb Z\left[\tfrac1N\right]\right),
\]
it follows that $\Theta_k$ has finite index in
$\SL(n,\mathbb Z\left[\tfrac1N\right])$.

Therefore, every finite-index subgroup $\Gamma$ contains a
two-generated finite-index subgroup. This proves the theorem.
\end{proof}
\begin{proof}[Proof of Theorem~\ref{Thm: SL(2,O_F) rank gradient}]
We separate the proof into the two cases $F=\mathbb Q$ and $F\neq \mathbb Q$.

\textbf{Case 1:} $F=\mathbb Q$. Let $N\geq2$ and let $\Gamma$ be a finite-index subgroup of $\SL(2,\mathbb Z\left[\tfrac1N\right])$. There exists $m\in \mathbb N$ coprime to $N$ such that for all $a\in m\mathbb Z\left[\tfrac1N\right]$,
\[
\begin{pmatrix}
   1 &a \\
   0 &1
\end{pmatrix}\in \Gamma, \quad \begin{pmatrix}
   1 &0 \\
   a &1
\end{pmatrix}\in \Gamma.
\]
Since $m$ is coprime to $N$, there exists $k\in \mathbb N$ such that $N^k\equiv 1\pmod{m^2}$. Let 
\[
D=\begin{pmatrix}
    m^{-\frac12} &0\\
    0 & m^{\frac{1}{2}}
\end{pmatrix}.
\]
Then 
\[
D\begin{pmatrix}
    1 & 0\\
    \frac{m}{N^k} & 1
\end{pmatrix}D^{-1}=B_{\tfrac{m^2}{N^k}}, \quad D\begin{pmatrix}
    1 & m\\
    0 & 1
\end{pmatrix}D^{-1}=A.
\]
By Lemma~\ref{Cor: t=kspm1}, $\frac{m^2}{N^k}\in \mathcal{CS}$. Therefore, 
\[
\left\langle\begin{pmatrix}
    1 & 0\\
    \frac{m}{N^k} & 1
\end{pmatrix},  \begin{pmatrix}
    1 & m\\
    0 & 1
\end{pmatrix}\right\rangle 
\]
is $\bigl(\{\infty\}\cup\Pi(N)\bigr)$-arithmetic. Hence, it is a finite-index subgroup of $\SL(2,\mathbb Z\left[\tfrac1N\right])$. 

\textbf{Case 2:} $F\neq \mathbb Q$. By assumption, $F$ is a non-CM field. We are in the first alternative of Lemma~\ref{lem: Sarma units}, and $\alpha$ is not a root of unity. Let $\Gamma$ be a finite-index subgroup of $\SL(2,\mathcal O_{F,S})$. Then there exists $m\in \mathbb N$ such that $B_m=\begin{pmatrix}
    1 & 0\\
    m &1
\end{pmatrix}\in \Gamma$. Since the image of $\alpha$ in $\mathcal O_{F,S}/m\mathcal O_{F,S}$ is a unit, there exists $k\in \mathbb N$ such that $\alpha^k-1\in m\mathcal O_{F,S}$. Let
\[
L'=\begin{pmatrix}
    \alpha^k &-\frac{\alpha^k-\alpha^{-k}}{m}\\
    0& \alpha^{-k}
\end{pmatrix}\in \SL(2,\mathcal O_{F,S}).
\]
Then there exists $k'\in \mathbb N$ such that
\[
L=(L')^{k'}=\begin{pmatrix}
    \alpha^{kk'} &-\tfrac{\alpha^{kk'}-\alpha^{-kk'}}{m}\\
    0& \alpha^{-kk'}
\end{pmatrix}\in \Gamma.
\]

Let $\Theta$ be the subgroup of $\Gamma$ generated by $B_m$ and $L$.

A direct computation shows that
\[
LB_m^{-1}LB_m=\begin{pmatrix}
    1 &\frac{2-2\alpha^{2kk'}}{m}\\
    0 &1
\end{pmatrix}.
\]

Since $L$ is upper triangular with diagonal entries $\alpha^{kk'}$ and $\alpha^{-kk'}$, conjugation by $L$ preserves $U^+$ and multiplies the upper-right parameter by $\alpha^{2kk'}$. Therefore, conjugating
\[
\begin{pmatrix}
1 &\frac{2-2\alpha^{2kk'}}{m}\\
0 &1
\end{pmatrix}
\]
by powers of $L$ and using Lemma~\ref{lem: Sarma units}, we conclude that $\Theta$ contains a finite-index subgroup of
\[
U^+(\mathcal O_{F,S})=\left\{
\begin{pmatrix}
1&x\\
0&1
\end{pmatrix}
:x\in\mathcal O_{F,S}
\right\}.
\]
On the other hand, $\Theta$ contains the nontrivial unipotent
\[
B_m=
\begin{pmatrix}
1&0\\
m&1
\end{pmatrix}.
\]
Hence, by \cite{MR2718935}, $\Theta$ is a finite-index subgroup of
$\SL(2,\mathcal O_{F,S})$.

This completes the proof.

\end{proof}

\section{Low-step relation numbers}\label{Sec: step 1 and 2}
In this section, we study the relation numbers of step at most two. The results obtained here will be used in the proof of Theorem~\ref{thm:Pell-equation}.

We begin with some notation.

\begin{defn}
Let $l\in\mathbb N$. A complex number $q\neq 0$ is called an \emph{$l$-step relation number} if there exist nonzero integers
\[
m_1,m_2,\ldots,m_{2l+1}
\]
such that
\[
B_q^{m_1}A^{m_2}\cdots A^{m_{2l}}B_q^{m_{2l+1}}\in U,
\]
and no such sequence exists for any smaller positive integer.

By convention, $0$ is the unique $0$-step relation number.
\end{defn}
In what follows, we will analyze which one- and two-step relation numbers give rise to arithmetic lattices via Theorem~\ref{Thm: characteristic S-arithmetic}.
\subsection{One-step relation numbers}
Let $q$ be a one-step relation number. By definition, there exist nonzero integers $m_1, m_2,m_3$ such that
\begin{equation}\label{eq: 1-step}
B_q^{m_1}A^{m_2}B_q^{m_3}=\begin{pmatrix}
    1+m_2m_3q&m_2\\
    q(m_1+m_3+m_1m_2m_3q)& 1+m_1m_2q
\end{pmatrix}\in U.
\end{equation}
Since $q\neq 0$, we have 
\[
q=-\tfrac{1}{m_1m_2}-\tfrac{1}{m_2m_3}.
\]
Substituting this expression for $q$ into equation~\eqref{eq: 1-step}, we obtain 
\[
\begin{pmatrix}
    -\tfrac{m_3}{m_1}& m_2\\
    0&-\tfrac{m_1}{m_3}
\end{pmatrix}\in U.
\]
By Theorem~\ref{Thm: characteristic S-arithmetic}, if
$m_1\neq m_3$, then $q\in\mathcal{CS}$. On the other hand, if $m_1=m_3$, then $q=-\tfrac{2}{m_1m_2}$, which is a rational number with numerator $1$ or $2$. Hence $q\in \mathcal{CS}$ by Corollary~\ref{Cor: t=kspm1}. 

We summarize the discussion in the following proposition.
\begin{proposition}\label{Prop: 1-step}
    The set of one-step relation numbers is
    \[
    \{\tfrac{1}{m}+\tfrac{1}{n}: m,n\in\mathbb Z\setminus\{0\}, m+n\neq 0\}.
    \]
    Moreover, all one-step relation numbers are in $\mathcal{CS}$.
\end{proposition}

\subsection{Two-step relation numbers}
Let $q$ be a two-step relation number. By definition, there exist nonzero integers $m_1, m_2,m_3, m_4,m_5$ such that 
\begin{equation}\label{eq: Step 2}
B_q^{m_1}A^{m_2}B_q^{m_3}A^{m_4}B_q^{m_5}=\begin{pmatrix}
    1+f_{11}(q)&m_2+m_4+m_2m_3m_4q\\
    f_{21}(q)& 1+f_{22}(q)
\end{pmatrix}\in U,
\end{equation}
where
\begin{itemize}
    \item $f_{11}(q)=(m_2m_3+m_4m_5+m_2m_5)q+m_2m_3m_4m_5q^2$;
    \item $f_{22}(q)=(m_1m_4+m_1m_2+m_3m_4)q+m_1m_2m_3m_4q^2$;
    \item $f_{21}(q)=(m_1+m_3)q+m_1m_2m_3q^2+m_5q(1+f_{22}(q))$.
\end{itemize} 
Equation~\eqref{eq: Step 2} will be the main topic of this section and the next. For convenience, we call this element $U(q;m_1,m_2,m_3,m_4,m_5)$.

The condition that this element belongs to $H$ is equivalent to 
\begin{enumerate}
    \item $f_{11}(q)=f_{22}(q)=f_{21}(q)=0$.
    \item\label{Cond: 2} $m_2+m_4+m_2m_3m_4q\in \mathbb Z$.
\end{enumerate}
By $f_{11}(q)=0$, we see that 
\[
q=-\left(\tfrac{1}{m_2m_3}+\tfrac{1}{m_3m_4}+\tfrac{1}{m_4m_5}\right).
\]
Substituting this value of \(q\) into \(f_{22}(q)=f_{21}(q)=0\), one obtains
\[
m_1m_2=m_4m_5.
\]
Moreover, condition~\eqref{Cond: 2} becomes
\[
m_2+m_4-m_2-m_4-\tfrac{m_2m_3m_4}{m_4m_5}
\in \mathbb Z.
\]
Using \(m_1m_2=m_4m_5\), this is equivalent to
\[
\tfrac{m_3m_4}{m_1}\in\mathbb Z,
\]
that is, \(m_1\mid m_3m_4\).

In summary, we have shown the following.

\begin{proposition}\label{prop: 2-step}
Let \(q\) be a two-step relation number, and let \(m_i\) be as above. Then
\[
q\in\mathcal{CS}
\]
if either
\[m_1m_2\neq m_4m_5
\qquad\text{or}\qquad
m_1\nmid m_3m_4.
\]
\end{proposition}
The following result is proved in \cite{MR1420342}. Let \(a,k,N\) be positive integers and let
\(\varepsilon_1,\varepsilon_2\in\{\pm1\}\) satisfy \(ka+\varepsilon_1\neq0\). Suppose that
\(s,t\) are nonzero integers satisfying the quadratic Diophantine equation
\begin{equation}\label{eq: quadratic in TT}
    at^2-(ka+\varepsilon_1)Ns^2=\varepsilon_2.
\end{equation}
Then $\tfrac{s^2}{t^2}$ is a relation number of step at most two, with
\[
m_1=\varepsilon_1\varepsilon_2t^2,\qquad
m_2=\varepsilon_1(ka+\varepsilon_1)N,\qquad
m_3=k,
\]
\[
m_4=-N,\qquad
m_5=-\varepsilon_2(ka+\varepsilon_1)t^2.
\]

Observe that
\[
m_1m_2=m_4m_5.
\]
Moreover, equation~\eqref{eq: quadratic in TT} implies that
\(\gcd(t^2,N)=1\). Hence the condition
\[
m_1\nmid m_3m_4
\]
is equivalent to
\[
t^2\nmid k.
\]
Therefore, Proposition~\ref{prop: 2-step} yields
\[
\tfrac{s^2}{t^2}\in\mathcal{CS}
\]
whenever \(t^2\nmid k\).

It remains to consider the case \(t^2\mid k\). Writing
\(k=t^2k'\), equation~\eqref{eq: quadratic in TT} becomes
\[
(a+k'aNs^2)t^2-\varepsilon_1Ns^2=\varepsilon_2,
\]
which is of the same form with \(k=0\). Hence
\(\tfrac{s^2}{t^2}\) is in fact a one-step relation number. By
Proposition~\ref{Prop: 1-step}, it follows that
\[
\tfrac{s^2}{t^2}\in\mathcal{CS}.
\]

We conclude that every relation number constructed in \cite{MR1420342}
belongs to \(\mathcal{CS}\).
\section{Constructing unipotent elements from quadratic forms}\label{Sec: quadratic forms to unipotents}

In this section, we prove Theorems~\ref{thm: quadratic form}
and~\ref{thm:Pell-equation}. The main idea is to determine when
equation~\eqref{eq: Step 2} produces an element of $U\setminus H$.
For the first theorem, we construct an upper-triangular element with
nontrivial diagonal entries. For the second theorem, we show directly
that the element produced by equation~\eqref{eq: Step 2} is upper
triangular.

\subsection{Proof of Theorem~\ref{thm: quadratic form}}

We begin with the following observation.

\begin{lemma}\label{lemma: 1/q^2}
Let
\[
q=\tfrac{s}{t}\in\mathbb Q\setminus\mathbb Z,
\qquad
\gcd(s,t)=1, \quad t>1,
\]
and suppose that there exist integers $m_1,m_2,m_3,m_4,m_5$ such that
\[
1+f_{22}(q)=\pm\tfrac{1}{t^2}.
\]
Then $q\in\mathcal{CS}$.
\end{lemma}

\begin{proof}
Since $f_{21}(q)$ is a cubic polynomial in $q$ and
$\Gamma_q\leq \Gamma_1^{(t)}(s)$,
there exists $k\in\mathbb Z$ such that
$f_{21}(q)=\tfrac{sk}{t^3}$.
Consequently,
\[
U(q;m_1,m_2,m_3,m_4,m_5)
\begin{pmatrix}
1&0\\
\mp kq&1
\end{pmatrix}
=
\begin{pmatrix}
*&*\\
0&\pm\tfrac{1}{t^2}
\end{pmatrix}.
\]
Because $t>1$, the resulting upper-triangular matrix has nontrivial
diagonal entries and therefore does not belong to $H$. Hence $U_q\neq H$.
Theorem~\ref{Thm: characteristic S-arithmetic} now implies that
$q\in\mathcal{CS}$.
\end{proof}

\begin{proof}[Proof of Theorem~\ref{thm: quadratic form}]
Let $A,B,C$ be as in the statement of the theorem, and let $s,t$ be
integers satisfying
\[
ABs^2+(A+B+C)st+t^2=\pm1.
\]

If $t=1$, then $q$ is integral and the conclusion follows from the
integral case. We may therefore assume that $t>1$.

Set
\[
d=\gcd(A,C)
\]
and define
\[
m_1=d,
\qquad
m_2=\tfrac{A}{d},
\qquad
m_3=\tfrac{Bd}{C},
\qquad
m_4=\tfrac{C}{d}.
\]
The number $m_3$ is an integer. Indeed, writing
\[
A=dA',
\qquad
C=dC',
\qquad
\gcd(A',C')=1,
\]
the assumption $C\mid AB$ implies that $C'\mid B$.

For any choice of $m_5$, we have
\[
m_1m_2m_3m_4=AB
\]
and
\[
m_1m_2+m_1m_4+m_3m_4=A+B+C.
\]
Therefore,
\[
m_1m_2m_3m_4s^2
+
(m_1m_2+m_1m_4+m_3m_4)st
+
t^2
=
\pm1.
\]
Dividing by $t^2$ and setting
$q=\tfrac{s}{t}$,
we obtain
\[
1+f_{22}(q)=\pm\tfrac{1}{t^2}.
\]
Lemma~\ref{lemma: 1/q^2} therefore gives
$q\in\mathcal{CS}$.
\end{proof}

\subsection{Proof of Theorem~\ref{thm:Pell-equation}}
Theorem~\ref{thm:Pell-equation} is proved by a more detailed analysis of equation~\eqref{eq: Step 2}. 

\begin{proof}[Proof of Theorem~\ref{thm:Pell-equation}]
Note that
\[
U(q;m_1,m_2,m_3,m_4,m_5)\in U
\]
if and only if \(f_{21}(q)=0\). Since \(q\) is a factor of
\(f_{21}(q)\), to find nonzero relation numbers it is enough to
consider the quadratic equation
\[
m_1m_2m_3m_4m_5q^2
+
\kappa q
+
m_1+m_3+m_5
=
0,
\]
where
\[
\kappa
=
m_1m_2m_3
+
m_1m_2m_5
+
m_1m_4m_5
+
m_3m_4m_5.
\]

This equation has rational roots if and only if its discriminant is a
square in \(\mathbb Q\). Since the discriminant is difficult to
analyze in its present form, we impose suitable relations among the
parameters \(m_i\).

Let \(k\neq 0,-1\) be an integer, and impose the constraints
\begin{equation}\label{eq: constraints}
m_1=km_3,
\qquad
m_2=k+1,
\qquad
m_4=k.
\end{equation}
Substituting these relations into the quadratic equation gives
\[
k^2(k+1)m_3^2m_5q^2
+
\bigl(
k(k+1)m_3^2
+
2k(k+1)m_3m_5
\bigr)q
+
(k+1)m_3+m_5
=
0.
\]
The discriminant of this equation is
\[
\Delta
=
k^2(k+1)^2m_3^2
\left(
m_3^2+\tfrac{4k}{k+1}m_5^2
\right).
\]

Let \(u,v\in\mathbb Z\) satisfy
\[
(k+1)u^2-kv^2=1.
\]
Set
\[
m_3=1
\qquad\text{and}\qquad
m_5=(k+1)uv.
\]
By Equation~\eqref{eq: constraints}, we then have
\begin{equation}\label{eq: mi}
m_1=k,
\qquad
m_2=k+1,
\qquad
m_3=1,
\qquad
m_4=k,
\qquad
m_5=(k+1)uv.
\end{equation}
Moreover,
\[
\Delta
=
k^2(k+1)^2
\left(
(k+1)u^2+kv^2
\right)^2.
\]
Therefore, the solutions of \(f_{21}(q)=0\) are \(q=0\) and
\[
q_\pm(u,v,k)
=
\tfrac{
-k(k+1)\bigl(1+2(k+1)uv\bigr)
\pm
k(k+1)\bigl((k+1)u^2+kv^2\bigr)
}{
2k^2(k+1)^2uv
}.
\]

Using the Pell equation, these two roots can be rewritten as
\[
q_+(u,v,k)
=
\tfrac{kv-(k+1)u}{k(k+1)u}
\]
and
\[
q_-(u,v,k)
=
-\tfrac{u+v}{kv}.
\]
The first root cannot vanish. Indeed, if
\[
kv=(k+1)u,
\]
then, since \(\gcd(k,k+1)=1\), there exists \(r\in\mathbb Z\) such
that
\[
u=kr,
\qquad
v=(k+1)r.
\]
Substitution into the Pell equation gives
\[
-k(k+1)r^2=1,
\]
which is impossible for \(k\neq0,-1\). On the other hand,
\(q_-(u,v,k)=0\) implies \(u=-v\), and the Pell equation then gives
\(u^2=1\), so \(uv=-1\). Consequently, the assumption \(uv\neq-1\)
implies that
\[
q_\pm(u,v,k)\neq0.
\]

Thus \(q_\pm(u,v,k)\) is a relation number of step at most two. If it
is a one-step relation number, Proposition~\ref{Prop: 1-step} implies
that
\[
q_\pm(u,v,k)\in\mathcal{CS}.
\]
Otherwise, it is a two-step relation number. Since
\[
m_1m_2=k(k+1)
\qquad\text{and}\qquad
m_4m_5=k(k+1)uv,
\]
the assumption \(uv\neq1\) gives
\[
m_1m_2\neq m_4m_5.
\]
Proposition~\ref{prop: 2-step} therefore implies that
\[
q_\pm(u,v,k)\in\mathcal{CS}.
\]

Finally, along a sequence of solutions for which \(u\to\infty\), the
Pell equation gives
\[
\tfrac{v^2}{u^2}
=
\tfrac{k+1}{k}-\tfrac{1}{ku^2}.
\]
Hence
\[
\tfrac{u}{v}
\longrightarrow
\pm\sqrt{\tfrac{k}{k+1}}.
\]
The result follows by a direct computation.
\end{proof}

\section{Proof of Theorem~\ref{Thm: small denominators}}\label{Sec: proof of theorem E}
We prove Theorem~\ref{Thm: small denominators} following the strategy of \cite{MR4505367}. Theorem~\ref{Thm: characteristic S-arithmetic} reduces the verification of $q\in \mathcal{CS}$ to the construction of a nontrivial element of $U_q$. The algorithm of Kim and Koberda provides a systematic way to produce words witnessing non-freeness. We show that these words satisfy the stronger conclusion required by Theorem~\ref{Thm: characteristic S-arithmetic}, thereby converting the non-freeness criterion of \cite{MR4505367} into an arithmeticity criterion.

Let us start by recalling the algorithm. 

For $x,y\in\mathbb Z$, define $(x;y)=\{x+ky: k\in\mathbb Z\}$.
\begin{defn}\cite[Definition 3.3]{MR4505367}\label{Def: s-good}
Let $s,w,m$ be nonzero integers such that $s>1$, and let
\[
D:=\gcd(w,sm);\quad d:=\gcd(w,s).
\]
We say that $(w;sm)\subseteq\mathbb Z$ is an \textbf{$s$-good residue class} if there exists an integer $y$ satisfying the following two conditions:
\begin{itemize}
    \item $yD\mid md$;
    \item $w\mid smy\pm D$.
\end{itemize}
In this case, $w$ is called a \textbf{good representative} of $(w;sm)$.
\end{defn}
Using the same strategy, we obtain the following variant of \cite[Lemma 3.5]{MR4505367}.  The significance of the lemma is that an $s$-good residue class produces an explicit element of $U_q\setminus H$, allowing us to apply Theorem~\ref{Thm: characteristic S-arithmetic}. Note that equation~\eqref{eq: s-good implies CS} is a reformulation of the argument used in the proof of \cite[Lemma 3.2]{MR4505367}.

\begin{lemma}\label{lemma: s-good to CS}
Let $(w;sm)$ be an $s$-good residue class with good representative $w$. Then
\[
\tfrac{s}{t}\in \mathcal{CS}
\]
for every $t\in (w;sm)$ that does not divide $w$.
\end{lemma}
\begin{proof}
  Let $D$ and $d$ be as in Definition~\ref{Def: s-good}. Set $w'=\tfrac{w}{D}$, $s'=\tfrac{s}{d}$ and $m'=\tfrac{md}{D}$. Suppose that $t\in (w;sm)$. Define $t'=\tfrac{t}{D}$ and $q'=\tfrac{s'}{t'}$. By the $s$-good hypothesis, there exists $y\in \mathbb Z$ such that
  \[
y\mid m', \quad w'\mid s'm'y\pm 1.
  \]
  Moreover, $t'\nmid w'$.

  Put $u_1=\tfrac{1\pm s'm'y}{w'}$ and $u_2=\tfrac{t'-w'}{s'm'}$. A direct computation shows that 
  \begin{equation}\label{eq: s-good implies CS}
  B_{q'}^{-t'(\mp y-u_1u_2)}A^{m'}B_{q'}^{-u_2}A^{-\tfrac{m'(\mp y-u_1u_2)}{\mp y}}B_{q'}^{\mp t'y}=\begin{pmatrix}
      1 & \tfrac{m'(\mp y-u_1u_2)w'}{\mp t'y}\\
      0&1
  \end{pmatrix}\in \Gamma_{q'}.
  \end{equation}
By definition, 
\[
t'u_1+s'm'(\mp y-u_1u_2)=1.
\]
Therefore, $\gcd(m'(\mp y-u_1u_2), t')=1$. Since $t'\nmid w'$, it follows that 
\[\tfrac{m'(\mp y-u_1u_2)w'}{\mp t'y}\notin \mathbb Z.\] 
Hence $q'\in\mathcal{CS}$ by
Theorem~\ref{Thm: characteristic S-arithmetic}. Since
\[
\tfrac{s}{t}=\tfrac{s'}{t'}\tfrac{d}{D},
\]
Corollary~\ref{cor: /n} implies that
$\tfrac{s}{t}\in\mathcal{CS}$.
\end{proof}
The main step in \cite{MR4505367} is the following lemma.
\begin{lemma}\cite[Lemma 4.3]{MR4505367}\label{lemma: enough s-good}
    Suppose that an integer $s$ satisfies $2\leq s\leq 27$ and $s\neq 24$. Then there exists a family of $s$-good residue classes
    \[
    \{(w_i,sm_i)\}
    \]
    whose union contains all integers that are relatively prime to $s$; moreover, we can require that $m_i\mid 60$ and $0\leq w_i< sm_i$. 
\end{lemma}
\begin{proof}
     All statements other than the additional requirement
$0\leq w_i<sm_i$ are proved in \cite{MR4505367}. The reduction to the range $0\leq w_i< sm_i$ is implemented in the proof code accompanying \cite[Proposition B.1]{MR4505367}: the command $wp = Mod[w, s m]$ replaces each residue class by its canonical representative modulo $sm_i$. See line 17 on page 3 of the \href{https://arxiv.org/src/1901.06375v4/anc/relnum-v2.pdf}{ancillary file} for \cite{MR4505367}.
\end{proof}

Lemmas~\ref{lemma: s-good to CS} and \ref{lemma: enough s-good} together show that, for a fixed value of $s$, it suffices to verify a finite collection of exceptional parameters. The list of $s$-good residue classes with $m\neq 1$ is given on pages 395--396 of the \href{https://arxiv.org/src/1901.06375v4/anc/relnum-v2.pdf}{ancillary file} accompanying \cite{MR4505367}. The required verifications are carried out in Appendix~\ref{Appen: verify}, completing the proof of Theorem~\ref{Thm: small denominators}.

\appendix
\section{Relations of rational numbers}\label{Appen: verify}
In this appendix, we verify the remaining cases required for Theorem~\ref{Thm: small denominators}. By Lemmas~\ref{lemma: s-good to CS} and \ref{lemma: enough s-good}, it suffices to check a finite collection of exceptional rational parameters arising from the $s$-good residue classes of \cite{MR4505367}.

Before proceeding, we make three observations. First, if $t\equiv \pm1 \pmod{s}$, then the result follows from Corollary~\ref{Cor: t=kspm1}. We therefore omit these cases. In particular, the result is already known for $|s|=1,2,3,6$.

Second, if $(w;sm)$ is an $s$-good residue class, then so is $(-w;sm)$. Moreover, $w$ is a good representative if and only if $-w$ is also
a good representative. For example,
$(17;24)=(-5;24)$.
Thus, once we have shown that $\tfrac{12}{5}\in \mathcal{CS}$, there is no need to consider $\tfrac{12}{17}$.

Third, for a fixed value of $s$, we list the representatives $w$ in increasing order. When testing a given $w$, all of its proper divisors have already been considered. Consequently, it suffices to verify the condition for $w$ itself. When there is more than one good representative, we take the one with the shortest sequence.

The exceptional numbers arising from $s$-good classes with $m>1$ are listed on page 400 of the \href{https://arxiv.org/src/1901.06375v4/anc/relnum-v2.pdf}{ancillary file} for \cite{MR4505367}. Note that some classes are repeated. For example, $(7;36)=(43;36)$, so $\tfrac{18}{7}\in \mathcal{CS}$ implies $\tfrac{18}{43}\in \mathcal{CS}$. In the following list, we ignore these repeats. 

In the following list, $B$ always stands for the matrix $B_q$. The list is extracted from the computation in \cite{MR4505367}, along with additional information. See the Sage code below. 

Some words produced by the algorithm do not fit on a single line. In other cases, the Sage code fails to terminate due to the repetition of a pair of numbers. For these words, we search for alternative words by brute force whenever possible.

All of the following equations have been verified using \href{https://sagecell.sagemath.org/?z=eJyNWE1v20gSvRvwf-ghEIAaSxb7mzRWB2kxhznsAsFk5hCNHGgsOmFWoWxRXicI9N-3qj_YzZaUWQGJJHb1q1dVr6pbnpMZ-bI-7Juv-du3Y7Jc0jGhK_hQ4PtqdH31K1g0m7o9NIdvHyJTBmvXV5v6kSzy59Hd9RWB174-vOzbE8QCEZ89ot21r9ebD_XXp_x1TBq_vyE3M0Ld50fyumxWZDYj2ffMGeDrMzB6vW3aTf01z44Zbg-LjsH79zlsvqF3n1ejMey48ai4uYkcfF6RpiXZzWToIdB4_dRsa3jyD7Kt2_x1RNbtxuy7bbpN87E55KMLOwdUHBEf_dN639UfXnf7DcYPqWlmkPHusHua_XvX1h5y3XbA91f75WNtvn3P5tkdmY9JtoB3TP4RUQPXpucaEYNgEZ00HWl3B4JOXCQ2xbgYmUf8gcMYUzbAMtsgccgp2fa4fjjs9kAU15ZouBoaxEXGV72NS51nZ_GAAkCepg0rC6kYZaO_c9LVaXzrpqvJH-vtS_3Lfr_b59nvLQiyfjjUG9J9-_LXbkvWB_K065pDs2vJm-4O_mXkDcmbsSE8GsVpgb2Q1fZABo4htiYVjwv1Pg3VI9hghx0Se0JZ_Dxzmbn3u7zFmVJHfk6j_lfTdU37kTxsd-Ydkgxon-qu6bLe7YkYUMX1fwFmfahzeAiku_r5pW4fevU-QxBv3-bdiEzx_eBK9AqPs-z2865pc7_ltnvaYis5E8StxyH1J5UfRcH2Vj_NTnV_Gu67TzVBHPK6tgl62H152taHevvNetmcRG3o-KAfPtUP__ngR2IS-ph4BXkSiP6CippdTJc17KU3mMg9XBRwDwkB-_U4ZEv6-1Bb2TNMC1ONqanFOFn2bMCqDyYx8X7BxH9MTTwdMPEfU5NN8_hY752rPpRJbx_Mj0kZ3u3TKqy32_xh3dWdz3aNNcYpuVz5zY8wjLbrv-otZPxCrXCSGZhBFruXLfby_1vvZEK6_a4JkXra7Ibq7frpqW43-XANX4YzjLfszzY7Xb0JFbsjGXz13360w2f7Dgxwy2GfW5rLUNvV6EcIvdxOEfrS_xghlP8MRqSNVTLQ4wawqTvp8nnX1XucBLbTE9296fw9pqk7mJ1wVGIYf7Yw04emMOBxjlgvcGpnaOUGlnsYkUuHhVfp9dU_57_98htq8foqz55ncsqysRwzaIhscf99wo5z-J8eF-4tg4XlMpvQqcjGGQXr1XiZFfB5IjK4OY0RE4E0AukBUL4wGBZpdM_yhf_IPS6dKsSdcFlNecBWKTZHbG6xuaN2z-bGzylVpg0_Pq0CW3gWQ5YIWXpIdTZoCvvHGQNLj1INMCrMSTUWw9ydI1TCXh7nrhwA0QLZ0CLQcXCLQaiB4IkDzSoDy6WaMhEIT3Bh4IoaV3SQTHYMpWHGGQul4727qJrBXoRsyaoQSINSJpQUvJjSSpWBjF1P2AjDps-hgKDTaEOYIBEoiIR89gGiamJENpWAyEDTDhEpy5j7SZWZRCCpYaeHhUcJqjao2qNqAOOLe4Gpmp8iioKa7BeyANZRQczKAJqbFPA-BVwd_6bYtmWqKVUhDUnHAKo0qCEN8mJnm_h5FP0kDZ9PlUFTQ6nzM2AUSyRjrdOkQsJQEz01Wvi-uVx4amtEYTOTMXRCVJpkyqAnG7VASBGL3HQYP3r1W3mMYt07GnFvyFBhqrTWDCqAEtRFwUWpK6VLEDStQPqiJ9lbDogqkwQVkiCtoMIEjtmGaWnyQBXsp1EiVJoHZfSqer1ScX42AV3bsWDvwdIeVTgGAavqpVQejewTpNJInsVTt0y0rk3UOlJlkagynu4LLwvXxGfOD5wJVWGzUkpW8kJL6NUyTs7EWSRMlGHSK1qZmHoFXpAijiDTfbTAeRGqjI8TB9o4iCuwcCWmx2FkbHDy3ksn2cszC8KWWonSThnKldKFkJRpNi0Zl_EB4O0G5EpThzJtQXQOR6EeeHafe-LszOzHDq0kd13KaAmVYJWCUvjc9x3rzBI62tDpp6tJlivK-SpAiHjEl0bAfREw8CEwpQYZjphwHPjWZ9F3WhX82EfDOSYW6FNGI-5Ts5B64NYD9x54lDnm0G3BdYEPdekdMVppSxtgVIgDRjKuJI6YdcS8I1X5IXyvXX36GR9PThhAdnZWlQbNsiqen7g48FMZaVShReVQCPAmQ4QMHc_PjE2cuP6h6E8yr3MdNKSiopZUqEozJZUb90ILqUo4MrWuSoGfpWQKhptiFa-Ek5Kv_mB7EpQyQZluj8dIfz8AamUQeXTjKZM-wIQschcbQESZxlseFZyXkpY0ujXCwEzpaENnMJ1ZmH3n1S64_qHYK7xmAmoZH9JiwPvS_VSaU1tzuFopEc2O9OiGO0dhfMBdNXIiTy9ANu749wI8SrGoxYp7M5TAdpIZTpT5jnGdWQ0780xjsgJzzIowU0IrBlf4kYXGd9c2KqDb4kubu9Al-JXBr2Ly1WC6nz8guVCXT0iAxawAbsiKEK6HbKelqCB6YU9eIfAAsKnwzM1i4oFZ5qynHg9DuIByK3aqfF5K64AXcsriKVWm2OaXDqO9BEV6vA9_9sUX-eQCy6gZrIAVBqsKpTNshZS9LpT5_cNxU6QKVaWgJnYaxW71wCp_TzT_u74Rx5MSmkzQQhWaxrlwmbi-Mn_hCX8HMr-2R_8DKp3_QQ==&lang=sage&interacts=eJyLjgUAARUAuQ==}{Sage}.
\begin{lstlisting}[caption={Sage code}]
def SRem(x, y):
    r = x - y * floor(x / y)
    if r == x + y or (r != x and abs(r) <= abs(y) / 2):
        return r
    else:
        return r - y


def one_step(t, s, a, b):
    bs = b * s
    as_ = a * s

    u = SRem(t, bs)
    v = SRem(as_, u)

    A = matrix(QQ, [[1, 1], [0, 1]])
    B = matrix(QQ, [[1, 0], [QQ(b)/QQ(a), 1]])

    M = B^v * A^u

    X = a * u
    Y = b * v

    g = gcd(X, Y)
    if X < 0:
        g = -g

    new_pair = (X/g, Y/g)

    return M, new_pair, u, v


def run_until(t, s, a, b):
    B = matrix(QQ, [[1, 0], [QQ(b)/QQ(a), 1]])

    total_matrix = B
    pair = (t, s)

    print("Initial pair:", pair)
    print("Initial matrix B =")
    print(total_matrix)
    print()

    step = 0

    while pair != (1, 0):
        step += 1

        M, pair, u, v = one_step(pair[0], pair[1], a, b)

        total_matrix = M * total_matrix

        print(f"Step {step}:")
        print(f"u = {u}, v = {v}")
        print(f"Matrix = B^{v} A^{u}")
        print("New pair =", pair)
        print("Cumulative matrix =")
        print(total_matrix)
        print()

    return total_matrix
\end{lstlisting}
\begin{center}
   \textbf{ List of exceptional numbers and an element in $U_q\setminus H$}
\end{center}
\hrule

$s=5$: $(3;5)=(-2;5)$. $q=\tfrac{5}{2}$. \[B^{-2}A^{-1}BA^{-1}B=\begin{pmatrix}
    -\tfrac{1}{4}& \tfrac{1}{2}\\
    0 &-4
\end{pmatrix}.\]
\hrule

$s=7$: $(-2;7)=(5;7)$. $q=\tfrac{7}{2}$. 
\[
B^{-2}A(B^{-1}A^{-1})^2(BA^{-1})^3B=\begin{pmatrix}
 \tfrac{1}{64}& -\tfrac{359}{32}\\
 0& 64
\end{pmatrix}.
\]
$(-3;7)=(4;7)$, $q=\tfrac{7}{3}$.
\[
B^3A^{-1}B^2A^{-2}BA^{-1}B=\begin{pmatrix}
    -\tfrac{1}{27}& -\tfrac{43}{9}\\
    0 &-27
\end{pmatrix}.
\]

\hrule

$s=8$: $(-3;8)=(5;8)$, $q=\tfrac{8}{3}$.
\[
B^6A^{-1}BA^{-1}B=\begin{pmatrix}
    \tfrac{1}{9} &\tfrac{2}{3}\\
    0 &9
\end{pmatrix}.
\]
\hrule

$s=9$: $(-2;9)=(7;9)$ and $q=\tfrac{9}{2}>4$.\\
$(-4;9)=(5;9)$, $q=\tfrac{9}{4}$.
\[
B^{-2}A^{-2}BA^{-1}B=\begin{pmatrix}
    -\tfrac{1}{8}& \tfrac{3}{2}\\
    0 &-8
\end{pmatrix}.
\]

\hrule

$s=10$: $(-3;10)=(7;10)$. $q=\tfrac{10}{3}$.
\[
B^6AB^{-2}ABA^{-1}B^2A^{-1}BA^{-1}BA^{-1}B=\begin{pmatrix}
    -\tfrac{1}{729}&-\tfrac{8356}{243}\\
    0& -729
\end{pmatrix}.
\]
\hrule

$s=11$: $(-2;11)=(9;11)$. $q=\tfrac{11}{2}>4$.\\
$(-3;11)=(8;11)$. $q=\tfrac{11}{3}$.
\[
B^3A^{-2}(BA^{-1})^2B^2A^2B^{-1}A^{-3}BA^{-1}B^{-1}A^{-1}(BA^{-1})^4B=\begin{pmatrix}
    \tfrac{1}{59049}&\tfrac{112465430}{19683}\\
    0& 59049
\end{pmatrix}.
\]
$(-4;11)=(7;11)$, $q=\tfrac{11}{4}$.
\[B^{-4}AB^2A^{-1}BA^{-1}B=\begin{pmatrix}
    -\tfrac{1}{32}& \tfrac{25}{8}\\
    0& -32
\end{pmatrix}.\]
$(5;11)=(-6;11)$, $q=\tfrac{11}{6}=\tfrac{1}{2}\tfrac{11}{3}$.

\hrule

$s=12$: $(5;24)=(19;24)$, $q=\tfrac{12}{5}$.
\[B^{-1}A^5B^{-1}A^{-1}BA^{-1}B=\begin{pmatrix}
    \tfrac{1}{25}& -\tfrac{57}{5}\\
    0& 25
\end{pmatrix}.\]
$(7;24)=(17;24)$, $q=\tfrac{12}{7}$.
\[B^{-7}A^3B^{4}A^{2}BABA^{-1}B=\begin{pmatrix}
    \tfrac{1}{2401}& -\tfrac{70501}{343}\\
    0& 2401
\end{pmatrix}.\]

\hrule

$s=13$: $(-2;13)=(11;13)$. $q=\tfrac{13}{2}>4$.\\
$(-3;13)=(10;13)$. $q=\tfrac{13}{3}>4$.\\
$(-4;13)=(9;13)$. $q=\tfrac{13}{4}$.
\[B^{36}A^{-1}BA^{-1}BA^{-1}B=\begin{pmatrix}
    -\tfrac{1}{64}& -\tfrac{9}{16}\\
    0& -64
\end{pmatrix}.\]
$(-5;26)=(21;26)$. $q=\tfrac{13}{5}$.
\[B^{-15}A^{-1}BA^{-1}B=\begin{pmatrix}
    -\tfrac{1}{25}& \tfrac{3}{5}\\
    0& -25
\end{pmatrix}.\]
$(-6;13)=(7;13)$. $q=\tfrac{13}{6}$.
\[B^{-2}A^{-3}BA^{-1}B=\begin{pmatrix}
    -\tfrac{1}{12}& \tfrac{5}{2}\\
    0& -12
\end{pmatrix}.\]
$(-8;26)=(18;26)$. $q=\tfrac{13}{8}=\tfrac{1}{2}\tfrac{13}{4}$.

\hrule

$s=14$: $(-3;14)=(11;14)$. $q=\tfrac{14}{3}>4$.\\
$(-5;14)=(9;14)$. $q=\tfrac{14}{5}$.
\[B^{10}A^{-2}B^{2}A^{-1}BA^{-1}B=\begin{pmatrix}
    -\tfrac{1}{125}& -\tfrac{114}{25}\\
    0& -125
\end{pmatrix}.\]

\hrule

$s=15$: $(-2;15)=(13;15)$. $q=\tfrac{15}{2}>4$.\\
$(-4;15)=(11;15)$. $q=\tfrac{15}{4}$.
\[
\begin{aligned}[c]
B^{-5}A^{-4}BA^{4}(BA^{-1})^2BAB^{-3}A^{-2}(B^{-1}A)^2
B^{-1}A^{-1}BA^{-2}(BA^{-1})^5B \\
=
\begin{pmatrix}
    \tfrac{1}{16777216}& -\tfrac{3700348796781}{4194304}\\
    0& 16777216
\end{pmatrix}.
\end{aligned}
\]
$(-8;15)=(7;15)$. $q=\tfrac{15}{8}=\tfrac{1}{2}\tfrac{15}{4}$.

\hrule

$s=16$: $(-3; 16)=(13;16)$, $q=\tfrac{16}{3}>4$.\\
$(-5;16)=(11;16)$. $q=\tfrac{16}{5}$.
\[
B^{15}A^{3}B^{-2}A^{-1}(BA^{-1})^2B=\begin{pmatrix}
    \tfrac{1}{625}& \tfrac{1616}{125}\\
    0& 625
\end{pmatrix}.
\]
$(-7;32)=(25;32)$. $q=\tfrac{16}{7}$.
\[
B^{14}A^{-2}BA^{-1}B=\begin{pmatrix}
    \tfrac{1}{49}& \tfrac{11}{7}\\
    0& 49
\end{pmatrix}.
\]
$(-9;32)=(23;32)$. $q=\tfrac{16}{9}$.
\[
B^{-18}A^{2}BA^{-1}B=\begin{pmatrix}
    \tfrac{1}{81}& -\tfrac{23}{9}\\
    0& 81
\end{pmatrix}.
\]

\hrule

$s=17$: $(-2;17)=(15;17)$. $q=\tfrac{17}{2}>4$.\\
$(-3;17)=(14;17)$. $q=\tfrac{17}{3}>4$.\\
$(-4;17)=(13;17)$. $q=\tfrac{17}{4}>4$.\\
$(-5;51)=(46;51)$. $q=\tfrac{17}{5}$.
\[
B^{-10}A^{-1}BA^{-1}B^3A^{-1}B^2ABA^{-2}B^{-1}A^5(BA^{-1})^3B=\begin{pmatrix}
    -\tfrac{1}{390625}& \tfrac{852830752}{78125}\\
    0& -390625
\end{pmatrix}.
\]
$(-6;17)=(11;17)$. $q=\tfrac{17}{6}$.
\[
B^{-6}A^{2}B^2A^{-1}B^2A^{-1}BA^{-1}B=\begin{pmatrix}
    \tfrac{1}{324}& -\tfrac{1057}{54}\\
    0& 324
\end{pmatrix}.
\]
$(-7;51)=(44;51)$. $q=\tfrac{17}{7}$.
\[
B^{14}ABA^{3}B^{-1}A^{-2}B^{-1}A^2B^{-2}A^{-1}B^5A^{-4}B^{-1}A^{-1}BA^{-1}B=\begin{pmatrix}
    -\tfrac{1}{5764801 }& -\tfrac{136670451272}{823543}\\
    0& -5764801 
\end{pmatrix}.
\]
$(-10;51)=(41;51)$. $q=\tfrac{17}{10}=\tfrac{1}{2}\tfrac{17}{5}$.\\
$(-12;51)=(39;51)$. $q=\tfrac{17}{12}=\tfrac{1}{2}\tfrac{17}{6}$.\\
$(-22;51)=(29;51)$. $q=\tfrac{17}{22}=\tfrac{1}{2}\tfrac{17}{11}$.\\
$(-24;51)=(27;51)$. $q=\tfrac{17}{24}=\tfrac{1}{4}\tfrac{17}{6}$.

\hrule

$s=18$: $(-5;36)=(31;36)=(67;36)=(103;36)$ and $(5;36)=(41;36)=(77;36)$; $(-5;54)=(49;54)$ and $(5;54)=(59;54)$. $q=\tfrac{18}{5}$.
\[
B^{10}A^{-1}B^6A^{7}B^{-1}A^{-1}B^{-1}A^3B^{-1}A^{2}(BA^{-1})^4B=\begin{pmatrix}
    -\tfrac{1}{1953125}& -\tfrac{21828329652}{390625}\\
    0& -1953125
\end{pmatrix}.
\]
$(-7;36)=(29;36)=(65;36)=(101;36)$ and $(7;36)=(43;36)=(79;36)$; $(-7;54)=(47;54)$ and $(7;54)=(61;54)$. $q=\tfrac{18}{7}$.
\[
B^{-14}AB^{-6}A^{-1}BA^{-1}B=\begin{pmatrix}
    \tfrac{1}{343}& -\tfrac{481}{49}\\
    0& 343
\end{pmatrix}.
\]
$(-11;54)=(97;54)$. $q=\tfrac{18}{11}$.
\[
B^{-1}AB^{-3}A^2B^{-1}AB^{-1903}=\begin{pmatrix}
    -1331& \tfrac{52}{121}\\
    0& -\tfrac{1}{1331}
\end{pmatrix}.
\]
$(-13;108)=(95;108)$. $q=\tfrac{18}{13}$.
\[
B^{-3}A^{-1}B^{-2}AB^{-1}A^{-270}B^{78}=\begin{pmatrix}
2197&\tfrac{3438}{169}\\
0 &\tfrac{1}{2197}
\end{pmatrix}.
\]
$(-23;54)=(85;54)$. $q=\tfrac{18}{23}$.
\[
B^{69}A^{-3}B^{7}A^{3}B^{-5}A^{-1}B=\begin{pmatrix}
-\tfrac{1}{12167}&-\tfrac{119977}{529}\\
0 &-12167
\end{pmatrix}.
\]
$(-25;108)=(83;108)$. $q=\tfrac{18}{25}=\tfrac{1}{5}\tfrac{18}{5}$.\\

\hrule

$s=19$: $(-2;19)=(17;19)$. $q=\tfrac{19}{2}>4$.\\
$(-3;19)=(16;19)$. $q=\tfrac{19}{3}>4$.\\
$(-4;19)=(15;19)$. $q=\tfrac{19}{4}>4$.\\
$(-5;19)=(14;19)$, $q=\tfrac{19}{5}$.
\[
\begin{aligned}[c]
B^{-5}A^{-1}B^{-1}AB^{53}A^{-1}B^{2}A^{3}BA^2B^{-1}A^{-1}BA^{-4}BAB^{-1}A^4BA^{-1}BAB^{-2}A^7(BA^{-1})^6B \\
=
\begin{pmatrix}
    \tfrac{1}{3814697265625}& -\tfrac{147456840177984456855266}{762939453125}\\
    0& 3814697265625
\end{pmatrix}.
\end{aligned}
\]
$(-6;19)=(13;19)$. $q=\tfrac{19}{6}$.
\[ (BA^{-1})^3B^{-1}A^{5}BAB^{8}A^{3}
B^{-1}ABA^{-1}B^{-8}A^{-1}B^{-1}A^{-5}
B(AB^{-1})^3=\begin{pmatrix}
    -\tfrac{1}{18}& \tfrac{143385181}{3}\\
    0&-18
\end{pmatrix}.\]
$(-7;38)=(31;38)$. $q=\tfrac{19}{7}$.
\[
B^{14}A^2B^3A^{-1}BA^{-1}B=\begin{pmatrix}
    \tfrac{1}{343}&\tfrac{437}{49}\\
    0 & 343
\end{pmatrix}.
\]
$(-8;38)=(30;38)$. $q=\tfrac{19}{8}$.
\[
B^{-2}A^4B^{-1}A^{-5}BA^{-2}BA^{-1}B=\begin{pmatrix}
    -\tfrac{1}{512}&\tfrac{7325}{64}\\
    0&-512
\end{pmatrix}.
\]
$(9;19)=(-10;19)$. $q=\tfrac{19}{10}$.
\[
B^{-2}A^5BA^{-1}B=\begin{pmatrix}
    \tfrac{1}{20}& -\tfrac{11}{2}\\
    0 &20
\end{pmatrix}.
\]
$(-11;38)=(27;38)$. $q=\tfrac{19}{11}$.
\[
B^{-1}A^3B^{-1}AB^{-3}AB^{1012}=\begin{pmatrix}
    1331 &\tfrac{92}{121}\\
    0&\tfrac{1}{1331}
\end{pmatrix}.
\]

\hrule

$s=20$: $(-3;20)=(17;20)$. $q=\tfrac{20}{3}>4$.\\
$(-7;20)=(13;20)$. $q=\tfrac{20}{7}$.
\[
B^{-2}AB^{-1}AB^{-1}A^{3}BAB^{203}
=
\begin{pmatrix}
2401 & \tfrac{1418}{343}\\
0 & \tfrac{1}{2401}
\end{pmatrix}.
\]
$(-9;40)=(31;40)=(71;40)$ and $(9;40)=(49;40)$. $q=\tfrac{20}{9}$.
\[
B^{-1}A^9B^{-1}A^{-2}BA^{-1}B=\begin{pmatrix}
    \tfrac{1}{81}& -\tfrac{346}{9}\\
    0 &81
\end{pmatrix}.
\]
$(-11;80)=(69;80)$. $q=\tfrac{20}{11}$.
\[
B^{-44}A^4BA^3BA^2BA^{-1}B=\begin{pmatrix}
    \tfrac{1}{14641}& -\tfrac{244272}{1331}\\
    0& 14641
\end{pmatrix}.
\]
$(-29;80)=(51;80)$. $q=\tfrac{20}{29}$.
\[
B^{-3}A^{2}B^{-1}A^{33}B^{-116}
=
\begin{pmatrix}
841 & -\tfrac{305}{29}\\
0 & \tfrac{1}{841}
\end{pmatrix}.\]

\hrule

$s=21$: $(-2;21)=(19;21)$, $q=\tfrac{21}{2}>4$.\\
$(-4;21)=(17;21)$, $q=\tfrac{21}{4}>4$.\\
$(-5;21)=(16;21)$, $q=\tfrac{21}{5}>4$.\\
$(-8;42)=(34;42)=(76;42)=(118;42)$ and $(8;42)=(50;42)=(92;42)$; $(-8;63)=(55;63)$ and $(8;63)=(71;63)$. $q=\tfrac{21}{8}$.
\[
B^{40}A^{-1}BA^{-1}B=\begin{pmatrix}
    \tfrac{1}{64}&\tfrac{5}{8}\\
    0 &64
\end{pmatrix}.
\]
$(-13;126)=(113;126)$. $q=\tfrac{21}{13}$.
\[
B^{-6}A^{1}B^{-1}A^{2}B^{-455}
=
\begin{pmatrix}
169 & -\tfrac{3}{13}\\
0 & \tfrac{1}{169}
\end{pmatrix}.\]
$(-29;126)=(97;126)$. $q=\tfrac{21}{29}$.
\[
B^{-1}AB^{29}A^{-4}B^{-4}A^{4}B^{-4}A^{-1}B
=
\begin{pmatrix}
\tfrac{1}{841} & \tfrac{106071}{29}\\
0 & 841
\end{pmatrix}.
\]
\hrule
\printbibliography
\end{document}